\documentclass[11pt,a4paper,draft]{amsart} 
\usepackage{amssymb,amscd,amsmath, young,mathrsfs,mathabx}
\usepackage{mathrsfs, mathdots,rotating,multirow} 
\usepackage[mathcal]{eucal} 
\usepackage{float}
\usepackage[usenames]{color}
\usepackage{soul}
\usepackage{longtable}

\theoremstyle{plain}
\newtheorem{thm}{Theorem}[section]
\newtheorem{lem}[thm]{Lemma}
\newtheorem{pro}[thm]{Proposition}
\newtheorem{cor}[thm]{Corollary}
\newtheorem*{claim*}{Claim}

\newtheorem{con}[thm]{Conjecture}

\newtheorem{qun}[thm]{Question}

\theoremstyle{remark}

\newtheorem{rem}[thm]{Remark}
\newtheorem{exm}[thm]{Example}

\newtheorem{dfn}[thm]{Definition}
\newtheorem{fact}[thm]{Facts}

\numberwithin{equation}{section}

\numberwithin{table}{section}

\newcommand{\N}{\mathbb{N}}
\newcommand{\Z}{\mathbb{Z}}
\newcommand{\Q}{\mathbb{Q}}

\newcommand{\tensor}{\otimes}

\newcommand{\mff}{\mathfrak{f}}

\renewcommand{\epsilon}{\varepsilon}

\renewcommand{\phi}{\varphi}
\renewcommand{\theta}{\vartheta}

\newcommand{\ideal}{\triangleleft}

\newcommand{\Zp}{\mathbb{Z}_{p}}
\newcommand{\bsm}{\boldsymbol{m}}

\DeclareMathOperator{\gr}{gr}

\DeclareMathOperator{\Gr}{Gr}

\def \idealgr {\triangleleft_\textup{gr}}

\def \idealn {\triangleleft_{n}}
\def \idealkn {\triangleleft_{k,n}}

\def \bfm {{\bf m}}

\def \mcL {\ensuremath{\mathcal{L}}}

\def \Fp {\ensuremath{\mathbb{F}_p}}

\def \Fq {\ensuremath{\mathbb{F}_q}}
\def \Zp  {\mathbb{Z}_p}

\author{Marcus du Sautoy} \address{Mathematical Institute, University of
	Oxford, Oxford OX2 6GG, Great Britain}
\email{dusautoy@maths.ox.ac.uk}

\author{Seungjai Lee} \address{The Research Institute of Basic Sciences, Seoul National University, Seoul 08826, Republic of Korea}
\email{seungjai.lee@snu.ac.kr}

\begin{document}
	
	\title{Uniformity in Higher class Free Lie algebras}

	\begin{abstract}
		Let $\mff_{c,2}$ denote a free class-$c$ Lie rings on $2$ generators. We investigate the zeta functions enumerating graded ideals in $\mff_{c,2}(\Fp)$ for $c\leq6$, prove that they are uniformly given by polynomials in $p$ for $c\leq5$ and not uniformly given by a polynomial in $p$ for $c=6$. We also show that the zeta functions enumerating one-step graded ideals $\mff_{c,2}(\Fp)$ is always given by a polynomial in $p$ for all $c$. 
	\end{abstract}
	
	\maketitle

	\thispagestyle{empty}

\section{Introduction}
\subsection{Background and Motivation}
\subsubsection{Enumerating ideals in Free Lie rings}
Let $L$ be a Lie ring, additively isomorphic to $\Z^{n}$ for some $n\in\N$, and for $m\in\N$  let 
\[a_{m}^{\triangleleft}(L):=|\{H\triangleleft L\mid|L:H|=m\}|.\]
In their seminal paper \cite{GSS/88}, Grunewald, Segal, and Smith defined the \textit{ideal zeta functions of $L$} to be the Dirichlet generating series
\[\zeta_{L}^{\triangleleft}(s):=\sum_{m=1}^{\infty}a_{m}^{\triangleleft}(L)m^{-s},\]
where $s$ is a complex variable. This admits a natural Euler decomposition $\zeta_{L}^{\triangleleft}(s)=\prod_{p\textrm{ prime}}\zeta_{L(\Zp)}^{\triangleleft}(s)$,
where 
\[\zeta_{L(\Zp)}^{\triangleleft}(s)=\zeta_{L\tensor\Zp}^{\triangleleft}(s)=\sum_{i=0}^{\infty}a_{p^{i}}^{\triangleleft}(L)p^{-is}.\] We call these the \textit{local ideal zeta functions of $L$}. It is proven \cite[Theorem 1]{GSS/88} that these local functions are rational functions in $p$ and $p^{-s}$.

One of the major questions raised in \cite{GSS/88} concerns the behavior of the rational functions $\zeta_{L(\Zp)}^{\triangleleft}(s)$ as $p$ varies: the \textit{uniformity} problem.
\begin{dfn}\label{dfn:zp.uniform}
	The ideal zeta function $\zeta_{L}^{\triangleleft}(s)$ is \textit{finitely uniform} if there exist rational functions $W_{1}(X,Y),\ldots,W_{k}(X,Y)\in\Q(X,Y)$ for $k\in\mathbb{N}$ such that, for almost all primes $p$, 
	\[\zeta_{L(\Zp)}^{\triangleleft}(s)=W_{i}(p,p^{-s})\]
	for some $i\in\{1,\ldots,k\}$. It is  \textit{uniform} if $k=1$, and \textit{non-uniform} if it is not finitely uniform.  
\end{dfn}
Based on the evidence of the small examples, in \cite{GSS/88} the authors speculated whether $\zeta_{L}^{\triangleleft}(s)$ is finitely uniform for any Lie ring $L$, which is now known to be false (cf. \cite{duS-ecI/01}).

However, in \cite{GSS/88} the authors had sufficient confidence about the zeta
functions for a certain class of Lie rings to make a following conjecture:
\begin{con}[\cite{GSS/88}]\label{con:uniformity}
	Let $\mff_{c,d}$ denote the free class-$c$ nilpotent Lie rings on $d$ generators defined over $\Z$. Then for any $c,d\in\N$, the ideal zeta function $\zeta_{\mff_{c,d}}^{\triangleleft}(s)$ is uniform, i.e., there exists a rational function $W(X,Y) \in \mathbb
	\Q(X,Y)$ so that for almost all primes
	\[
	\zeta_{\mff_{c,d}(\Zp)}^{\triangleleft}(s) = W(p,p^{-s}).
	\]
\end{con}
In fact, \cite[Theorem 2]{GSS/88} proved that $\zeta_{\mff_{2,d}}^{\triangleleft}(s)$ is uniform for all $d\in\N$. However, for higher classes (e.g. $c>2$ and $d\geq2$), the only known result regarding the uniformity of $\zeta_{\mff_{c,d}}^{\triangleleft}(s)$ is that of $\mff_{3,2}$, the free class-3 Lie ring on 2 generators, which is proved to be uniform by Woodward \cite{Woodward/05}. Computing $\zeta_{\mff_{3,3}}^{\triangleleft}(s)$ or $\zeta_{\mff_{4,2}}^{\triangleleft}(s)$ still look out of reach at the moment.

What makes $\mff_{c,d}$ particularly interesting objects is their connection with an enumerative problem in finite group theory. Let $f_{n}(p)$ denote the number of isomorphism classes of $p$-groups of order $p^{n}$. In 1960, Higman \cite{HigmanII/60} conjectured that for a fixed $n$, there exists a fixed integer $N$ and finitely
many polynomials $g_{i}(x)\;(i=1,2,\ldots,N)$ such that if
$p\equiv i\bmod N$ then 
\[
f_{n}(p)=g_{i}(p).
\]
This is now known as Higman's PORC (Polynomial On Residue Classes) Conjecture. This conjecture is currently known to be true for $n\leq7$, but still remains open for $n\geq8$ (see \cite{BEO/02}, \cite{VL/2012}, \cite{VL/2015}, or \cite{LVL22} for a detailed account with recent works).

Connection between the uniformity of ideal zeta functions of nilpotent Lie rings and Higman's PORC conjecture was first observed by the first author, where he conjectured \cite[Conjecture 5.11]{duS/02} that proving the uniformity of $\zeta_{\mff_{c,d}}^{\triangleleft}(s)$ for all $c,d\in\N$ might contribute to prove Higman's PORC conjecture. In particular, in \cite[Conjecture 5.11]{duS/02} he showed that the problem
of enumerating finite $p$-groups  may be reduced to enumerating ideals
in $\mff_{c,d}$, up to the action of its algebraic automorphism group. Currently several explicit  observations (e.g, \cite{DuSVL/2012,Lee/2016,LEE2022}) have been made to support the connection between the uniformity of ideal zeta functions and Higman's PORC conjecture.
\subsubsection{Graded ideal zeta functions of Free Lie rings}
One of the main obstacles to understand the nature of $\zeta_{\mff_{c,d}(\Zp)}^{\triangleleft}(s)$ and to answer Conjecture \ref{con:uniformity} is that they are in general too difficult to compute. In this light, in \cite{LeeVoll/18} the second author and Voll investigated the \textit{graded ideal zeta functions} of $\mff_{c,d}$.

Let $L$ be a nilpotent Lie
ring of nilpotency class $c$, free of finite rank over $\Z$, with
lower central series $(\gamma_{i}(L))_{i=1}^{c}$.   For $i \in
\{1,\dots,c\}$, set $L_i:=\gamma_{i}(L)/\gamma_{i+1}(L)$.  An ideal $I$ of $L$ (of
finite index in $L$) is \emph{graded} if
$I=\bigoplus_{i=1}^{c}(I\cap L_i)=I_{1}\oplus\cdots\oplus I_{c}$, where $I_i=I\cap L_{i}$. In this case we write
$I\triangleleft_{\gr}L$. We define the \emph{graded ideal zeta
	function of $L$} enumerating graded ideals in $L$ of finite index in $L$ as the Dirichlet generating series
\begin{equation}\label{def:graded.ideal.z.f.}
	\zeta_{L}^{\idealgr}(s)=\sum_{I\triangleleft_{\gr}L}|L:I|^{-s},
\end{equation}
where $s$ is a complex variable. $\zeta_{L}^{\idealgr}(s)$ converges on a complex half-plane, and it yields  a natural Euler decomposition $\zeta_{L}^{\idealgr}(s)=\prod_{p\textrm{ prime}}\zeta_{L(\Zp)}^{\idealgr}(s)$. For a similar definition in a slightly more general setting, see \cite[Section~3.1]{Rossmann/16} or \cite[Section 1.1]{LeeVoll/18}.

Note that $\zeta_{L}^{\idealgr}(s)$ can be seen as ``approximations'' of $\zeta_{L}^{\triangleleft}(s)$. Indeed, almost all Euler factors
$\zeta^{\idealgr}_{L(\Zp)}(s)$ actually enumerate a sublattice
of the lattice of ideals enumerated by
$\zeta^{\triangleleft}_{L(\Zp)}(s)$. In \cite{LeeVoll/18} the authors proved that $\zeta_{\mff_{c,d}}^{\idealgr}(s)$ is uniform for $\mff_{2,d}$ and $(c,d)=\{(3,2), (3,3), (4,2)\}$, extending the results from the ``ungraded'' setting.

\subsubsection{Ideal zeta functions over $\Fp$}
Let $\Fq$ denote a finite field on $q=p^r$ elements, where $p$ is a prime, and let $A$ be a finite-dimensional $\Fq$-algebra, additively isomorphic to $\Fq^{n}$. Let $a_{q^{i}}^{\ideal}(A)$ denote the number of ideals of $A$ of codimension $i$. One can analogously define the \textit{ideal zeta functions of $\Fq$-algebras} as a finite Dirichlet polynomial
\begin{equation}\label{eq:fp.zeta}
	\zeta_{A}^{\ideal}(s):=\sum_{B\ideal A}|A:B|^{-s}=\sum_{i=0}^{n}a_{q^{i}}^{\ideal}(A)q^{-is},
\end{equation}
where $s$  is a complex variable.

\begin{exm}
	For $n\in\N$, we have
	\begin{align*}
		\zeta_{\Fq^{n}}^{\ideal}(s)&=\sum_{i=0}^{n}a_{q^{i}}^{\ideal}(\Fq^{n})q^{-is}=\sum_{i=0}^{n}\binom{n}{i}_{q}\,q^{-is},
	\end{align*}
	where 
	\[\binom{n}{i}_{q}=\begin{cases}\frac{(1-q^n)(1-q^{n-1})\cdots(1-q^{n-i+1})}{(1-q)(1-q^2)\cdots(1-q^i)}&i\leq n,\\0&i>n,\end{cases}\] is the Gaussian binomial coefficient that counts the number of subspaces of (co)dimension $i$ in a vector space of dimension $n$ over $\Fq$. 
\end{exm}

Suppose now $q=p$  and $A=L(\Fp):=L\otimes\Fp$, where $L$ is a Lie ring of rank $n$ as defined earlier. In this setting one can view $\zeta_{L(\Fp)}^{\ideal}(s)$ as the ``simplest finite approximations" of $\zeta_{L(\Zp)}^{\ideal}(s)$ (which is why we define them as Dirichlet polynomials), in the sense that $\Fp=\Z/p\Z$ can be seen as a ``first layer" of $\Zp={\varprojlim}\Z/p^{i}\Z$, and all the ideals enumerated by $\zeta_{L(\Fp)}^{\ideal}(s)$ are also enumerated by $\zeta_{L(\Zp)}^{\ideal}(s)$. 

Let us analogously define the $\Fp$-uniformity as follows:
\begin{dfn}
	The ideal zeta function $\zeta_{L}^{\triangleleft}(s)$ is $\Fp$-\textit{finitely uniform}  if there exist polynomials $W_{1}^{\triangleleft}(X,Y),\ldots,W_{k}^{\triangleleft}(X,Y)\in\mathbb{Q}[X,Y]$ for $k\in\mathbb{N}$ such that, for almost all primes $p$, 
	\[\zeta_{L(\Fp)}^{\triangleleft}(s)=W_{i}^{\triangleleft}(p,p^{-s})\]
	for some $i\in[k]$, $\Fp$-\textit{uniform} if $k=1$, and $\Fp$-\textit{non-uniform} if it is not $\Fp$-finitely uniform.
\end{dfn}
In this sense, the ``classic'' uniformity problem as defined in Definition \ref{dfn:zp.uniform} can be viewed as a $\Zp$-uniformity. In \cite{Lee/20arxiv1}, the second author initiated the study of the zeta functions of $\Fp$-algebras. The results computed in \cite{Lee/20arxiv1} suggested that $\zeta_{L(\Fp)}^{\ideal}(s)$ is a good approximation of $\zeta_{L(\Zp)}^{\ideal}(s)$ for the uniformity problem, in the sense that there is no known example of a Lie ring $L$ whose $\Fp$-uniformity is different to its $\Zp$-uniformity yet.
\subsection{Main results and organizations}
In this article, we aim to further simplify the setting and investigate the uniformity question for graded ideal zeta functions of $\mff_{c,2}(\Fp)$, the free class-$c$ nilpotent $\Fp$-Lie algebras on 2 generators.

The  main results in this article are the following theorems: 
\begin{thm}\label{thm:uniform}
For $c\leq 5$, the graded ideal zeta function $\zeta_{\mff_{c,2}(\Fp)}^{\idealgr}(s)$
is uniformly given by a polynomial in $p$ for almost all primes. In other words, $\zeta_{\mff_{c,2}}^{\idealgr}(s)$ is $\Fp$-uniform for $c\leq5$.	
\end{thm}
\begin{thm}\label{thm:fin.uniform}
The graded ideal zeta function $\zeta_{\mff_{6,2}(\Fp)}^{\idealgr}(s)$ is not uniformly given by a polynomial in $p$. In particular, $\zeta_{\mff_{6,2}}^{\idealgr}(s)$ is not $\Fp$-uniform.
\end{thm}
Note that for $d=2$, as stated earlier, currently the $\Zp$-uniformity is only known for 
\begin{itemize}
	\item $\zeta_{\mff_{2,2}}^{\ideal}(s)$ (\cite{GSS/88}) and $\zeta_{\mff_{3,2}}^{\ideal}(s)$ (\cite{Woodward/05}), for ungraded setting, and
	\item $\zeta_{\mff_{2,2}}^{\idealgr}(s),\,\zeta_{\mff_{3,2}}^{\idealgr}(s)$, and $\zeta_{\mff_{4,2}}^{\idealgr}(s)$ (\cite{LeeVoll/18}), for graded setting.
\end{itemize}
Recently, the second author also showed in \cite[Theorem 3.7]{Lee/20arxiv1} that  $\zeta_{\mff_{2,2}}^{\ideal}(s)$,  $\zeta_{\mff_{3,2}}^{\ideal}(s)$, and  $\zeta_{\mff_{4,2}}^{\ideal}(s)$ are $\Fp$-uniform.  Theorem \ref{thm:uniform} further extends our database, and suggests that $\zeta_{\mff_{5,2}}^{\idealgr}(s)$, $\zeta_{\mff_{4,2}}^{\ideal}(s)$, and $\zeta_{\mff_{5,2}}^{\ideal}(s)$ may also be $\Zp$-uniform. 

More interestingly, Theorem \ref{thm:fin.uniform} is the first example of zeta functions regarding $\mff_{c,d}$ that turns out to be not uniform. Although we cannot compute $\zeta_{\mff_{6,2}}^{\idealgr}(s)$ or $\zeta_{\mff_{6,2}}^{\ideal}(s)$ yet, we believe it would be really surprising if any of them turns out to be $\Zp$-uniform.

Furthermore, in this article we also introduce and study the following zeta function counting some special graded ideals, namely the \textit{$n$-step graded ideals}. 
\begin{dfn}\label{def:nsg}
	Let $L$ be a nilpotent Lie algebra of class $c$. For $n\leq c$ and $k\in[c-n]$, an ideal $I\triangleleft L$ is said to be an $n$-step graded ideal, if it is of the form  \[I=\underbrace{0\oplus\cdots\oplus0}_{k-1}\oplus \underbrace{I_{k}\oplus I_{k+1}\oplus\cdots\oplus I_{k+n}}_{n+1}\oplus L_{k+n+1}\oplus\cdots\oplus L_{c},\] where $I_{k}\leq L_{k},I_{k+1}\leq L_{k+1},\ldots, I_{k+n}\leq L_{k+n}$. In this case we write $I\idealkn L$.

	Define the $n$-step graded ideal zeta function of L to be 
	\[\zeta_{L}^{\idealn}(s):=\sum_{k=1}^{c-n}\sum_{I\idealkn L}|L_{k}:I_{k}|^{-s}\cdots|L_{k+n}:I_{k+n}|^{-s}.\]
\end{dfn}
In Section \ref{sec:nsg}, we prove the following result:
\begin{thm}\label{thm:1step.uniform}
	The one-step graded ideal zeta function $\zeta_{\mff_{c,2}(\Fp)}^{\ideal_1}(s)$ of $\mff_{c,2}(\Fp)$ is always uniformly given by a polynomial in $p$ for all $c\in\N$.
\end{thm}
Note that this is a powerful result in the sense that it does not have any condition on nilpotency class $c$.

This article is arranged as follows: In Section \ref{sec:prelim}, we discuss some preliminary results from the study of vector spaces over $\Fp$ that allow us to prove our main theorems later. In Section \ref{sec:fc2} we study the free Lie algebras on 2 generators in some generality -- there is a large body of theoretical work concerning
free Lie algebras which we briefly recall. We go on to establish the setting
in which we have found it most useful to work. This seems to be a new way of
looking at these algebras and we are able to discover some new structural
results that are useful in the work that follows. In Section \ref{sec:cleq5} we use the arguments developed in previous sections to prove the $\Fp$-uniformity result on free class-$c$ Lie algebras on 2 generators for $c\leq5$. In Section \ref{sec:nsg} we then prove some $\Fp$-uniformity result on a general one-step graded ideals of $\mff_{c,2}$ and on some special two-step graded ideals. In Section \ref{sec:c=6} we prove another main result in this article, that $\zeta_{\mff_{6,2}(\Fp)}^{\idealgr}(s)$ is not $\Fp$-uniform. In Section \ref{sec:further} we end this article with some questions.

\subsection{Notation}

Given $n\in\N=\{1,2,\dots\}$, we write $[n]$ for $\{1,2,\dots,n\}$. Given a subset
$I\subseteq\N$, we write $I_{0}$ for $I\cup\{0\}$. For the Gaussian binomial coefficient $\binom{n}{i}_{p}$ where $n,i\in\Z$, we use the convention that
\[\binom{n}{i}_p=\begin{cases}
	\frac{(1-p^n)(1-p^{n-1})\cdots(1-p^{n-i+1})}{(1-p)(1-p^2)\cdots(1-p^i)}&0<i\leq n,\\
	1&\textrm{ if } i=0,\\
	0&\textrm{ if } i<0\textrm{ or }i>n.
\end{cases}\]

By a class of Lie ring we always mean the nilpotent class. For a prime $p$, we write $t=p^{-s}$, where $s$ is a complex variable.

\section{Vector spaces, Grassmannians, and matrix spaces over $\Fp$}\label{sec:prelim}

Let $K$ be a field. The \textit{Grassmannian} $\Gr_{K}(n,m)$ is the set of all $m$-dimensional subspaces of $K^{n}$. Every element of $\Gr_{K}(n,m)$ can be described as a span of some $m$ independent row vectors of length $n$, which we can arrange in an $m\times n$ matrix. 

\begin{exm}
	Let $K$ be a field and $V=\langle e_1,\ldots,e_5\rangle$ is an $5$-dimensional vector space over $K$. Let $W$ be a subspace of $V$ generated by $a_1e_1+a_3e_3,b_2e_2,c_4e_4+c_5e_5$, then one can represent $W$ in $V$ by 
	\begin{align*}
		\begin{pmatrix}
			a_1 & 0 & a_3 & 0 & 0\\
			0 & b_2 & 0 & 0 & 0 \\
			0 & 0 & 0 & c_4 & c_5 \\
		\end{pmatrix}.
	\end{align*}
\end{exm}

In fact, we can perform elementary row operations on the matrix without
changing the element of the Grassmannian it represents. Hence each subspace can be represented by a unique full-rank $m\times n$ matrix in Reduced Row Echelon Form (RREF). The subset of the Grassmannian where its elements have a particular RREF constitutes a \textit{Schubert Cell}, and $\Gr_{K}(n,m)$ is a disjoint union of Schubert Cells. In this paper we also frequently represent our spaces as row spaces of matrices, sometimes in RREF. By abusing notation, if $W$ is a subspace of $V$ and $M=M_W$ is the corresponding matrix whose row space gives $W$, then we will write $W=M$.

Since we count subspaces of a vector space of finite dimensions over $\Fp$ satisfying certain conditions, let us now concentrate on the case where $K=\Fp$. Here we record some  uniformity results on the number of subspaces over $\Fp$ that will be used later in this article. Let $V$ be a vector space of dimension $d$ over $\Fp$.

\begin{fact}\label{fact:poly}
	\begin{enumerate}
		\item The number of $m$-dimensional subspaces $U$ of $V$ is the Gaussian binomial coefficients
		\[\binom{d}{m}_{p}\]
		\item Let $W$ be a subspace of $V$ of dimension $l$. The number of $m$-dimensional subspaces $U$ of  $V$ that contain $W$ is
		\[\binom{d-l}{m-l}_p.\]
		\item Let $W$ be a subspace of $V$ of dimension $l$. The number of $m$-dimensional subspaces $U$ of $V$ that are disjoint to $W$ is
		\[p^{lm}\binom{d-l}{m}_{p}.\]
		\item Let $W$ be a subspace of $V$ of dimension $l$. The number of $m$-dimensional subspaces $U$ of  $V$ that intersect $W$ in
		dimension $s$ is
		\[p^{(l-s)(m-s)}\binom{d-l}{m-s}_p\binom{l}{s}_p.\]
	\end{enumerate}
	In particular, they are all uniformly given by a polynomial in $p$.
\end{fact}

\begin{pro}\label{pro:subspace.intersection}
Let $W_1$ and $W_2$ be subspaces of dimension $l_1$ and $l_2$ respectively in $V$, such that $W_1\cap W_2=0$. Let $g(d,l_1,l_2,n_1,n_2,m)$ denote the number of $m$-dimensional subspaces $U$ of $V$ that intersect $W_1$ in dimension $n_1$ and $W_2$ in dimension $n_2$. We have
\begin{align*}
	g(d,l_1,l_2,n_1,n_2,m)&=\binom{l_1}{n_1}_p\binom{l_2}{n_2}_p\sum_{r=0}^{\min\{l_1-n_1,l_2-n_2,m-n_1-n_2\}}\left(p^{(l_1-n_1+l_2-n_2-r)(m-n_1-n_2-r)}\right.\\
	&\left.\cdot\binom{l_1-n_1}{r}_p\binom{l_2-n_2}{r}_p\binom{d-l_1-l_2}{m-n_1-n_2}_p\prod_{e=0}^{r-1}(p^r-p^e)\right)
\end{align*}
\end{pro}
\begin{proof}
First, let $U\cap W_1=U_1$ and $U\cap W_2=U_2$. We know there are $\binom{l_1}{n_1}_p$ possibilities for $U_1$ and $\binom{l_2}{n_2}_p$ possibilities for $U_2$. Once we fix them, we can work in the quotient space $V/(U_1+U_2)$. Hence we can reduce to the case when $U_1=U_2=0$.

So assume $U_1=U_2=0$, and consider $U\cap(W_1+W_2)$. Suppose this has dimension $r$ with basis of the form $\{w_{1,i}+w_{2,i}\}_{i\in[r]}$, where $u_{1,i}\in W_1$ and $w_{2,i}\in W_2$ for $i\in[r]$. Since $U_1=U_2=0$,  $\{w_{1,1},\ldots,w_{1,r}\}$ is a basis for a subspace $X_1\leq W_1$ and $\{w_{2,1},\ldots,w_{2,r}\}$ is a basis for a subspace $X_2\leq W_2$. Note that for each fixed $r$, where $0\leq r \leq \min\{l_1-n_1,l_2-n_2,m-n_1-n_2\}$, there are  $\binom{l_1-n_1}{r}_p\binom{l_2-n_2}{r}_p$ possible choices for $r$-dimensional subspaces $X_1$ and $X_2$. 

Having chosen $X_1$ and $X_2$, pick a fixed basis $w_{1,i}$'s for $X_1$. Then we can choose any basis $w_{2,i}$'s of $X_2$, and each choice would generate different $r$-dimensional subspace $U\cap(W_1+W_2)$. Hence for a fixed $X_1$ and $X_2$ there are $\prod_{e=0}^{r-1}(p^r-p^e)$ possibilities for $U\cap(W_1+W_2)$.

Now, similar to Fact \ref{fact:poly} (3), we know that to extend a basis for $U\cap W_1+W_2$ to a basis for $U$ we have additional $p^{(l_1-n_1+l_2-n_2-l)(m-n_1-n_2-r)}\binom{d-l_1-l_2}{m-n_1-n_2}_p$ possibilities. Hence we get the final formula
\begin{align*}
	g(d,l_1,l_2,n_1,n_2,m)&=\binom{l_1}{n_1}_p\binom{l_2}{n_2}_p\sum_{r=0}^{\min\{l_1-n_1,l_2-n_2,m-n_1-n_2\}}\left(p^{(l_1-n_1+l_2-n_2-r)(m-n_1-n_2-r)}\right.\\
	&\left.\cdot\binom{l_1-n_1}{r}_p\binom{l_2-n_2}{r}_p\binom{d-l_1-l_2}{m-n_1-n_2}_p\prod_{e=0}^{r-1}(p^r-p^e)\right)
\end{align*}
as required.
\end{proof}
\section{Free Lie algebras on 2 generators}\label{sec:fc2}
Let $X=\{X_1,X_2\}$ be an alphabet on letters $X_1$ and $X_2$. Let $X^*$ denote the set of all words on $X$, and let $1$ denote the empty word in $X^*$. Given any word \[w=X_{i_{1}}^{n_{i_{1}}}\cdots X_{i_{l}}^{n_{i_{l}}}\in X^*,\] we write $|w|=\sum_{j=1}^{l}n_{i_{j}}$ and $X_{n}^*:=\{w\in X^*\mid |w|=n\}$. For $i\in[2]$ we also write $|w|_{i}$ to be the total number of occurrences of $X_{i}$ in the word $w$.  

Let $R$ denote a commutative ring with a unit element. Let $R\langle X\rangle$ denote a free $R$-module on $X^*$, and $R\langle X\rangle_{n}$ denote a free $R$-module on $X_{n}^*$. Elements of $R\langle X\rangle$ are of the form $a=\sum_{w\in X^*}\alpha_{w}\cdot w$, and elements of $R\langle X\rangle_{n}$ are of the form $a=\sum_{w\in X_{n}^*}\alpha_{w}\cdot w$, 
where $\alpha_{w}\in R$.

For any polynomials $P,Q$ in $R\langle X\rangle$, define their Lie bracket by 
\[[P,Q]=PQ-QP.\]
This defines a structure of Lie algebra on $R\langle X \rangle$. 

Let $\mcL_R(X)$ (or simply $\mcL(X)$) denote the smallest Lie subalgebra of $R\langle X\rangle$ that contains each letter in $X$. A Lie polynomial is an element of $\mcL_R(X)$. It is a well-known result (e.g., \cite[Theorem 5]{Reutenauer03}) that in fact $\mcL_R(X)$ is the free Lie algebra on $X$ over $R$. For instance, $\mcL_\Z(X)$ is the free Lie ring on 2 generators.

Let us also define a map (sometimes called the `Lie bracketing from right to left')
\[\chi:R\langle X\rangle \rightarrow \mcL_R(X)\]
by $\chi(1)=0$ and 
\[X_{i_{1}}\cdots X_{i_{n}}\mapsto [X_{i_{l}},[\cdots,[X_{i_{n-1}},X_{i_{n}}]\cdots]].\] 

To simplify the notation we will write $L$ for the free Lie object on $2$ generators, regardless of the field/ring we are working over. Define $L_{k}$ to be the subring of $L$ generated by Lie polynomials of weight $k$; this is the $k$-th term of the lower central series of $L$. We have
\[L=L_{1}\oplus L_{2}\oplus\cdots,\]
and in fact $\mff_{c,2}=L/L_{c+1}$. Also note that for $a\in L_{k}$, $\chi(a)=ka$. The number of generators $d_{k}$ of $L_{k}$ is given by the Witt function
\[d_{k}=\frac{1}{k}\sum_{j\mid k}\mu(j)2^{k/j},\]
where $\mu$ is the M\"{o}bius function.

For $i\in[2]$, let us define
\begin{align*}
	\phi_{i}:&L\rightarrow L\\
	&w\mapsto [X_{i},w]:=X_{i}\cdot w- w\cdot X_{i}
\end{align*}
to be a commutation with the generator $X_{i}$, and for $k\in[c-1]$ let 
\[\phi_{i,k}:=\left.\phi_{i}\right|_{L_{k}}:L_{k}\rightarrow L_{k+1}.\] 
We can of course define these maps on the full module $R\langle X \rangle$, and by abusing notation we will use the same notation to denote the map $w\mapsto X_{i}\cdot w- w\cdot X_{i}$ in both contexts.

From the freeness of $L$ we have the following simple injectivity of $\phi_{i}$.

\begin{lem}\label{lem:kernel} With regard to the maps $\phi_i:R \langle X \rangle \rightarrow R \langle X \rangle$ we have the following descriptions of their kernels: 
	\[
	\ker(\phi_i) = \langle 1,X_i,X_i^2,\ldots \rangle_R.
	\]
\end{lem}

\begin{proof}
	Without loss of generality let $i=1$ and suppose 
	that $a\in \ker(\phi_1).$ Write $a=\sum_{j=1}^n \alpha_jw_j$ with $\alpha_j \in R$ and without loss of generality we may 
	assume that all the $w_j$ are distinct. Then $\phi_1(a)=0$ tells us that 
	\begin{equation}\label{eq1}
		\sum_{j=1}^n \alpha_j(X_1\cdot w_j) = \sum_{j=1}^n\alpha_j (w_j\cdot X_1). 
	\end{equation}
	Suppose that there exists $j$ so that $|w_j|_2 >0.$ Then it is clear from (\ref{eq1}) that $w_j =X_1^{r_j}\cdot X_2 \cdot u_{j} $ for some $r_j >0$ and $u_{j}\in R\langle X \rangle.$ 
	Pick some $w_j$ so that $r_j$ is maximal. Without loss of generality we may assume that $j=1.$ Then the LHS of (\ref{eq1}) 
	contains $\alpha_1X_1\cdot w_1 = \alpha_1X_1^{r_1+1}\cdot X_2\cdot u_{1} $. This term does not cancel as all the $w_j$ are distinct 
	and so it also appears on the RHS of (\ref{eq1}). Thus there exists $l$ so that 
	\[
	w_l\cdot X = X_1^{r_1+1}\cdot X_2 \cdot u_{1} 
	\]
	which gives us that $w_l =X_1^{r_1+1}\cdot X_2\cdot \cdots .$ This contradicts the maximality of the choice of $r_1$ and the claim is proved.
\end{proof} 
\begin{cor}
	Except for $k=1$, where $\phi_{i,1}(X_i)=0$ for all $i$, the maps $\phi_{i,k}$ for $2\geq k \geq c-1$ are injections. 
\end{cor}
Note that we also have
\[L_{k+1}=\phi_{1,k}(L_{k})+\phi_{2,k}(L_{k}).\]

In order to establish structural results about $L$ we must introduce some new concepts.

\begin{dfn}	\label{def:permutations} Define a permutation $\sigma_1:X^\ast \rightarrow X^\ast$ by
	the following actions: $ 1\mapsto 1, X_2^i \mapsto X_2^i, X_1\cdot w\mapsto
	w\cdot X_1, X_2^j\cdot X_1\cdot w\mapsto w\cdot X_1\cdot X_2^j.$ Similarly
	define $\sigma_2.$ We write $\Gamma_n^1$ for the group of permutations of
	$X_n^\ast$ generated by $\sigma_1$ and similarly for $\Gamma_n^2.$ We also
	define $\Gamma_n$ for the group generated by $\sigma_1$ and $\sigma_2.$

\end{dfn}

\begin{rem}There may be interest in these groups in their own right. For instance we
have some calculations showing that
\[
|\Gamma_1| = 1, \quad |\Gamma_2| = 2, \quad |\Gamma_3| = 36, \quad
|\Gamma_4| = 96, \quad |\Gamma_5| = 329204736000000.
\]
Also the $\Gamma_1, \ldots,\Gamma_4$ are solvable whereas $\Gamma_5$ has a
composition series
\[
\Gamma_5 = N_1 \geq N_2 \geq \cdots \geq N_7 = 1
\]
where
\[
N_4/N_5 = A_{10} = N_5/N_6.
\]
\end{rem}
As a consequence of the Definition \ref{def:permutations} we have 

\begin{lem}	\label{lemma2}
	\begin{enumerate}
	\item Let $w \in X_n^\ast.$ Then there exists $g \in \Gamma_{n+2}$ so that
	$g(wX_1X_2)=wX_2X_1$;
	\item let $a \in R\langle X \rangle_n.$ Then $aX_1-X_1\sigma_1(a) \in \phi_2(R\langle X \rangle_n).$
	\end{enumerate}

\end{lem}

\begin{proof}
For the first part, notice that
\[
wX_1X_2 \stackrel{\sigma_1^{-1}}{\rightarrow} X_2X_1w
\stackrel{\sigma_2}{\rightarrow} X_1wX_2\stackrel{\sigma_1}{\rightarrow}
wX_2X_1.
\]
For (2), we show that this holds for any word $w \in X_n^\ast.$ There are three cases
to consider: (i): $w = X_1u, u \in X_{n-1}^\ast;$ (ii):
$w=X_2^n;$ (iii): $ w = X_2^iX_1u, u \in X_{n-i-1}^\ast.$
In case (i), we have
\[
wX_1 - X_1\sigma_1(w) = X_1uX_1-X_1uX_1 = 0 \in \phi_2(R \langle X \rangle_n);
\]
in case (ii) we have
\begin{align*}
	wX_1-X_1\sigma_1(w) & = X_2^nX_1 - X_1X_2^n \\
	& =  \sum_{j=1}^n \phi_2(X_2^{n-j}X_1X_2^{j-1});
\end{align*}
and in case (iii) we have
\begin{align*}
	wX_1-X_1\sigma_1(w) & =  X_2^iX_1uX_1 - X_1uX_1X_2^i \\
	& = \sum_{j=1}^i \phi_2(X_2^{i-j}X_1uX_1X_2^{j-1}).
\end{align*}
This completes the proof.
\end{proof}

In the light of Lemma \ref{lemma2} we make the following definition:

\begin{dfn}	\label{defn2}
	For every $n\geq 1$ set
	\[
	M_n^{X_1X_2} := \{a\in R \langle X \rangle_n: \phi_1(a) \in \phi_2(R \langle X \rangle_n)\}
	\]
	and similarly define $M_n^{X_2X_1}.$

\end{dfn}

We give the following description of these sets:

\begin{pro} 	\label{prop1} For every $n\geq 1$ we have
\begin{enumerate}
	\item $M_n^{X_1X_2} = R \langle X \rangle_n^{\Gamma_n^1}$
	\item $M_n^{X_2X_1} = R \langle X \rangle_n^{\Gamma_n^2}$
	\item $M_n^{X_1X_2}\cap M_n^{X_2X_1} = R \langle X \rangle_n^{\Gamma_n}$
\end{enumerate}
	where $R \langle X \rangle_n^{\Gamma_n^1}$, $R \langle X \rangle_n^{\Gamma_n^2}$, and $R \langle X \rangle_n^{\Gamma_n}$ denote the elements of $R \langle X \rangle_n$ that are left
	fixed by the permutation group $\Gamma_n^1$, $\Gamma_n^2$, and $\Gamma_n$, respectively.

\end{pro}
\begin{proof}
It is sufficient to prove (1). Suppose first that
\[
a \in R \langle X \rangle_n^{\Gamma_n^1}.
\]
Then $\sigma_1 (a) = a$ and so
\begin{align*}
	aX_1-X_1\sigma_1(a) & = aX_1-X_1a \\
	& =  \phi_1(a) \in \phi_2(R \langle X \rangle_n)
\end{align*}
by Lemma \ref{lemma2}. Hence $a\in M_n^{X_1X_2}$

Now suppose that
\[
a \in M_n^{X_1X_2}.
\]
Then there exists $b$ so that
\[
aX_1-X_1a = X_2b-bX_2
\]
and by Lemma \ref{lemma2} there exists $c$ so that
\[
aX_1-X_1\sigma_1(a) = X_2c-cX_2.
\]
Subtracting we see that
\[
X_1(\sigma_1(a)-a) = X_2(b-c) - (b-c)X_2
\]
which forces
\[
\sigma_1(a) = a
\]
as required.
\end{proof}

We are now in a position to prove the result we need in order to consider
the question of uniformity for the free  nilpotent Lie algebras of arbitrary class on 2 generators.

\begin{thm} 	\label{thm1}
	For any $a \in M_n^{X_1X_2}\cap M_n^{X_2X_1}$ we have $\chi(a)=0$.

\end{thm}

\begin{proof}
To prove this it is sufficient to consider sentences all of whose
words $w$ have $|w|_1=k, |w|_2=l$ with $l+k=n$ say. The cases where one of
$k,l$ are zero are easily dealt with and so we will from hereon assume that
$k,l\geq 1.$ Suppose then that
\[
w \in \Delta := \sum_{|w|_1=k,|w|_2=l}w.
\]
We know by Proposition \ref{prop1} that given any $\sigma\in \Gamma_n,$ we have
$\sigma(\Delta) =\Delta$ and so $w\in \Delta \implies \sigma(w)\in \Delta$
also. We claim that $\Gamma_n$ acts transitively on the set $B=\{w \in
A_n^\ast:|w|_1=k,|w|_2=l\}.$

To prove this claim let $w_1,w_2 \in B.$ We wish to find $\sigma$ so that
$\sigma(w_1)=w_2.$  Without loss of generality we may assume that $w_2
=X_1^kX_2^l$ and then again without loss of generality that
$w_1=X_1^{a_1}X_2^{b_1}\cdots X_1^{a_r}X_2^{b_r}$ where $a_i,b_i>0$ for all
$i.$ We induct on $r.$ If $r=1$ the claim is trivially true. Considering the
general case notice that
\[
X_1^{a_1}X_2^{b_1}\cdots X_1^{a_r}X_2^{b_r} \stackrel{\sigma_2}{\rightarrow}
X_2^{b_1-1}X_1^{a_2}\cdots X_2^{b_r+1}X_1^{a_1}
\stackrel{\sigma_1^{-a_1}}{\rightarrow} X_1^{a_1}X_2^{b_1-1}\cdots
X_2^{b_r+1}.
\]
Repeating this process $b_1$ times one sees it is possible to map
\[
X_1^{a_1}X_2^{b_1}\cdots X_1^{a_r}X_2^{b_r} \rightarrow
X_1^{a_1+a_2}X_2^{b_2}\cdots X_1^{a_r}X_2^{b_1+b_r} :=
X_1^{c_1}X_2^{d_1}\cdots X_1^{c_{r-1}}X_2^{d_{r-1}}
\]
and by induction the proof of the claim is completed.

We have shown that given any $\alpha \in R \langle X \rangle_n^{\Gamma_n}$ there exist
$\alpha_0, \ldots ,\alpha_n \in K$ so that
\[
\alpha = \sum_{i=0}^n \alpha_iP_i(X_1,X_2), \textnormal{ where }
P_i(X_1,X_2) = \sum_{|w|_1 =i, |w|_2=n-i}w
\]
To complete the proof of the theorem it is therefore sufficient to show that
for any $i$ we have $\chi(P_i(X_1,X_2))=0.$ Now for any $w \in P_i$ the right
hand end of $w$ is one of the following four options: i) $X_1^m$, ii)
$X_2^m$, iii) $X_1X_2$, iv) $X_2X_1$ where in cases i) and ii) $m$ is
greater than $1.$ In cases i),ii) it is clear that $r(w)=0.$ Now notice that
for any $w$ as in case iii), writing $w= vX_1X_2$ then there exists
$\overline{w}:=vX_2X_1 \in P_i$ also. The proof is completed by noticing
that $\chi(w)=-\chi(\overline{w}).$
\end{proof}

\section{Counting graded ideals of $\mff_{c,2}(\Fp)$ for $c\leq 5$}\label{sec:cleq5}
For $c\in\N$, let $I=I_{1}\oplus\cdots\oplus I_{c}$ denote the graded ideal of $L=\mff_{c,2}(\Fp)$. For $k\in[c-1]$ and $I_k \subset L_{k}$, define a map  $\phi(I_k):=\phi_{1}(I_k)+\phi_{2}(I_k)$. Note that $\phi(L_k)=L_{k+1}$ and $I\idealgr L$ implies $\phi(I_k)\leq I_{k+1}$ for all $k\in[c-1]$.

For $\boldsymbol{m}=(m_{k})_{k\in[c]}\in\N_{0}^{c}$, write
\[b_{\boldsymbol{m}}(L)=\{I\idealgr L\mid \dim(I_1)=m_1,\ldots,\dim(I_c)=m_c\}\]
and
\[b_{m}(L)=\{I\idealgr L\mid \dim(I)=m\}=\sum_{\{\bsm|m_1+\cdots+m_c=m\}}b_{\bsm}(L).\]
By definition we have
\begin{align*}
	\zeta_{L}^{\idealgr}(s)&=\sum_{m=0}^{n}b_{m}(L)t^{n-m}\\
	&=\sum_{m=0}^{n}\sum_{\{\bsm|m_1+\cdots+m_c=m\}}b_{\bsm}(L)t^{n-m}\\
	&=\sum_{m_1=0}^{d_1}\cdots\sum_{m_c=0}^{d_c}b_{\bsm}(L)t^{n-(m_1+\cdots+m_c)}.
\end{align*}

Normally, there is no simple description for $\sum_{m_1=0}^{d_1}\cdots\sum_{m_c=0}^{d_c}b_{\bsm}(L)$. However, for $c\leq5$ we can explicitly write down $\sum_{m_1=0}^{d_1}\cdots\sum_{m_c=0}^{d_c}b_{\bsm}(L)$ easily.

For $k\in[c-1]$, let $d_{\phi}^{(k)}(I_{k}):=\dim(\phi(I_k))$ denote the dimension of $\phi(I_k)$ in $L_{k+1}$. The crucial observation is that, for $c\leq5$, the dimension $m_{k}=\dim(I_k)$ completely determines  $d_{\phi}^{(k)}(I_{k})$. In other words, for any $m_{k}$-dimensional subspaces $I_{k}$ and $I_{k}'$ in $L_{k}$, the dimensions of their images $\phi(I_{k})$ and $\phi(I_{k}')$ in $L_{k+1}$ are the same.

\begin{lem}\label{lem:dimension}
	Let $I=I_1\oplus\cdots\oplus I_5$ be a graded ideal of $\mff_{5,2}(\Fp)$.
	\begin{enumerate}
		\item	For $k\in[4]$, if $m_{k}=0$ then $d_{\phi}^{(k)}(I_{k})=0$. 
		\item If $m_1=1$ or $2$, then $d_{\phi}^{(1)}(I_{1})=1$.
		\item If $m_2=1$, then $d_{\phi}^{(2)}(I_{2})=2$.
		\item If $m_3=2$, then $d_{\phi}^{(3)}(I_{3})=3$, and if $m_3=1$, then $d_{\phi}^{(3)}(I_{3})=2$.
		\item If $m_4=3$, then $d_{\phi}^{(4)}(I_{4})=6$, if $m_4=2$, then $d_{\phi}^{(4)}(I_{4})=4$, and if $m_4=1$, then $d_{\phi}^{(4)}(I_{4})=2$.		
	\end{enumerate}
\end{lem}
\begin{proof}
	The only non-trivial result is (5), which follows from the fact that in $L_{5}$, $\phi_{1}(L_4)\cap\phi_{2}(L_4)=\{0\}$, which forces $d_{\phi}^{(4)}(I_{4})=2m_{4}$.
\end{proof}

For $k\in[4]$, Lemma \ref{lem:dimension} allows us to abuse the notation and write $d_{\phi}^{(k)}(m_k)=d_{\phi}^{(k)}(I_{k})$. Hence we can write down the formula for $\sum_{m_1=0}^{d_1}\cdots\sum_{m_c=0}^{d_c}b_{\bsm}(L)$ systematically using the `top-down'-approach.

\begin{thm}
	For $c\leq5$, we have 
	\begin{align*}
		\zeta_{\mff_{c,2}(\Fp)}^{\idealgr}(s)&=\sum_{m_1=0}^{d_1}\cdots\sum_{m_c=0}^{d_c}b_{\bsm}(\mff_{c,2}(\Fp))t^{d-(m_{1}+\cdots+m_{c})}\\
		&=\sum_{m_1=0}^{d_1}\cdots\sum_{m_{c}=0}^{d_c}\binom{d_1}{m_{1}}_{p}\binom{d_{2}-d_{\phi}^{(1)}(m_1)}{m_2-d_{\phi}^{(1)}(m_1)}_{p}\cdots\binom{d_{c}-d_{\phi}^{(c-1)}(m_{c-1})}{m_c-d_{\phi}^{(c-1)}(m_{c-1})}_{p}t^{d-(m_1+\cdots+m_c)}
	\end{align*}
\end{thm}
\begin{proof}
	We start from $I_{1}$. For each $0\leq m_1\leq d_1$, there are exactly $\binom{d_1}{m_1}_{p}$ possible subspaces $I_{1}$ of dimension $m_1$ in $L_{1}$.
	
	Once we choose our $I_{1}$, we need to have $\phi(I_{1})\leq I_{2}\leq L_2$. Since by Lemma \ref{lem:dimension} the dimension of $\phi(I_{1})$ is $d_{\phi}^{(1)}(m_1)$, the number of such $I_{2}$ of dimension $m_{2}$ is simply $\binom{d_2-d_{\phi}^{(1)}(m_1)}{m_{2}-d_{\phi}^{(1)}(m_1)}_{p}$. 	One can repeat this process until we choose $I_c$ of dimension $m_{c}$.
\end{proof}

With this we have:
\begin{thm}\label{thm:cleq5}
	\begin{align*}
		\zeta_{\mff_{2,2}(\Fp)}^{\idealgr}(s)&=1+\binom{2}{1}_{p}t+t^2+t^3,\\
		\zeta_{\mff_{3,2}(\Fp)}^{\idealgr}(s)&=1+\binom{2}{1}_{p}t+t^2+t^3+\binom{2}{1}_{p}t^4+t^5,\\
		\zeta_{\mff_{4,2}(\Fp)}^{\idealgr}(s)&=1+\binom{2}{1}_{p}t+t^2+t^3+\binom{2}{1}_{p}t^4+\left(1+\binom{2}{1}_p\right)t^5+\binom{3}{2}_{p}t^6+\binom{3}{1}_{p}t^{7}+t^{8},\\	
		\zeta_{\mff_{5,2}(\Fp)}^{\idealgr}(s)&=1+\binom{2}{1}_{p}t+t^2+t^3+\binom{2}{1}_{p}t^4+\left(1+\binom{2}{1}_p\right)t^5+\left(\binom{3}{2}_{p}+\binom{2}{1}_p\binom{2}{1}_p\right)t^6\\
		&+\left(\binom{2}{1}_p+\binom{3}{2}_p\binom{2}{1}_p+\binom{3}{1}_{p}\right)t^{7}+\left(\binom{3}{2}_p+\binom{3}{2}_p\binom{4}{3}+1\right)t^{8}\\
		&+\left(\binom{3}{1}_p\binom{4}{2}_p+\binom{6}{5}_p\right)t^{9}+\left(\binom{3}{1}_p\binom{4}{1}_p+\binom{6}{4}_p\right)t^{10}+\left(\binom{3}{1}_p+\binom{6}{3}_p\right)t^{11}\\
		&+\binom{6}{2}_{p}t^{12}+\binom{6}{1}_{p}t^{13}+t^{14}.			
	\end{align*}
In particular, this proves Theorem \ref{thm:uniform}.
\end{thm}
\section{Counting one-step and two-step graded ideals in free Lie algebras on 2 generators}\label{sec:nsg}
In this section, by using the combinatorial results we have proved in Section \ref{sec:fc2} about the maps $\phi_1$ and $\phi_2$, we prove another uniformity results on general one-step graded ideals and some special two-step graded ideals of $\mff_{c,2}(\Fp)$. 

Recall Definition \ref{def:nsg}. An ideal $I \triangleleft L$ is said to be a one-step graded ideal if it is of the form
\[
I = 0\oplus\cdots0\oplus I_k \oplus I_{k+1} \oplus L_{k+2} \oplus \cdots\oplus L_c \quad \text {
	for some } 1\leq k\leq c-1.
\]

Note that in this case the codimension of $I$ in $L$ is given by the sum
	of the codimension of $I_k$ in $L_k$ with the codimension of $I_{k+1}$ in
	$L_{k+1}.$

Throughout this section, to ease the notation we let $L=\mff_{c,2}(\Fp)$ for $c\in\N$ and assume $p\geq c$ (as we only want to prove results about all but finitely many primes $p$ we can make the latter assumption). Given a subspace $I_{k}\leq L_{k}$, let us define the \textit{dimension of collapse} of $I_{k}$, denoted by $\dim_{col}(I_k)$, to be $\dim(\phi_1 (I_k) \cap \phi_2 (I_k))$. Then we have
\begin{align*}
	\dim (\phi_1 (I_k)+\phi_2 (I_k)) &= \dim(\phi_1(I_k)) +\dim(\phi_2(I_k))- \dim_{col}(I_k)\\
	&= 2\dim (I_k) - \dim_{col}(I_k).
\end{align*}

For any $1\leq k \leq c-1$, define $\Lambda_{k+1} := \phi_1(L_k) \cap \phi_2(L_k)\leq L_{k+1}$,
and write $\dim(\Lambda_{k+1})=l_{k+1}$. By the injectivity of $\phi_{i}$'s we can also define \[W_{k,1}:=\phi_{1}^{-1}(\Lambda_{k+1}),\,W_{k,2}:=\phi_{2}^{-1}(\Lambda_{k+1}),\] where $\dim(W_{k,1})=\dim(W_{k,2})=l_{k+1}$. For each $k$, subspaces $\Lambda_{k+1}$, $W_{k,1}$, and $W_{k,2}$ are uniquely determined from the structure of $L$. We prove that $W_{k,1}$ and $W_{k,2}$ are also disjoint.

\begin{lem}
	$W_{k,1}\cap W_{k,2}=\{0\}$ for all $1\leq k\leq c-1$ and $p\geq c$.
\end{lem}
\begin{proof}
		Suppose $a\in W_{k,1}\cap W_{k,2}$. Then both $\phi_1(a),\phi_2(a)\in\Lambda_{k+1}$, which implies
	$\phi_1(a)\in\phi_2(L_k)$ and $\phi_2(a)\in\phi_1(L_k)$. This means $a\in M_n^{X_1X_2}\cap M_n^{X_2X_1} \cap L_k$, and by the definition of $\chi$ and Theorem \ref{thm1}, we have $\chi(a)=ka=0$. For $p\geq c$ and $1\leq k\leq c-1$, as $k\neq0$  we have $a=0$ as required.
\end{proof}

Now suppose $I_{k}$ is an $m$-dimensional subspace of $L_{k}$. Let us write $U_{k,1}:=I_{k}\cap W_{k,1}$ and $U_{k,2}:=I_{k}\cap W_{k,2}$. Then $\dim(\phi(I_{k}))=2m-i$ if and only if $\dim_{col}(I_k)=i$, which is equivalent to say $\dim(\phi_{1}(U_{k,1})\cap\phi_{2}(U_{k,2}))=i$. Using this we prove the following result:

\begin{thm} \label{thm:dim.col2.poly}
	For $1<k \leq c-1$, let $f_k(m,i,p)$ denote the number of spaces $I_k \leq L_k$ such that $\dim I_k
	=m$ and  $\dim_{col}(I_k)=i$. Then $f_k(m,i,p)$ is uniformly given by a polynomial in $p$.
\end{thm}

\begin{proof}
	
	We have described how a space $I_k$ of dimension $m$ in $L_k$ maps to a
	space of dimension $2m-i$ in $L_{k+1}$  We count $f_k(m,i,p)$ as follows:
	
	Fix $j$ so that $i\leq j \leq l_{k+1}$. For each such $j$ there are $\binom{l_{k+1}}{j}_{p}$ spaces $U_1$ of dimension $j$ in
	$W_{k,1}$. To each such space $U_1$ we can associate
	in a natural way a
	disjoint $j$-dimensional space
	\[
	U_2:= \phi_2^{-1}(\phi_1(U_1)) \leq W_{k,2}.
	\]
	We can write
	\[
	f_k(m,i,p) = \sum_{j=i}^{l_{k+1}}
	h_k(m,j,i,p)\binom{l_{k+1}}{j}_{p},
	\]
	where $h_k(m,j,i,p)$ is the number of $m$-dimensional spaces in $L_k$ that
	intersect $W_{k,1}$ in a fixed
	$j$-dimensional space and map to a $(2m-i)$-dimensional space in $L_{k+1}.$
	
	Now, to compute $h_k(m,j,i,p)$, we must count those $m$-dimensional spaces $I_{k}$ in $L_{k}$  which meet a $l_{k+1}$-dimensional space $W_{k,1}$ in $U_1$
	and meet the disjoint $j$-dimensional space $U_2$ in a space of
	dimension $i$. Note that this number is equal to the number of $m-i$-dimensional subspaces $I_{k}'$ in  $L_{k}/U_{1}$ which is disjoint to a $l_{k+1}-j$-dimensional space $W_{k,1}/U_1$ and intersect $j$-dimensional space $U_2+U_1/U_1$ in dimension $i$. By Proposition \ref{pro:subspace.intersection}, we know this is equal to  $g(d_k,l_{k+1}-j,j,0,i,m)$, which is uniformly given by a
	polynomial in $p$. Hence 
	\begin{align*}
		f_k(m,i,p)&=\sum_{j=i}^{l_{k+1}}
		h_k(m,j,i,p)\binom{l_{k+1}}{j}_{p}\\
		&=\sum_{j=i}^{l_{k+1}}g(d_k,l_{k+1}-j,j,0,i,m)\binom{l_{k+1}}{j}_{p}\\
		&=\sum_{j=i}^{l_{k+1}}\binom{l_{k+1}}{j}_{p}\binom{j}{i}_p\sum_{r=0}^{\min\{l_{k+1}-j,j-i,m-j-i\}}\left(p^{(l_{k+1}-i-r)(m-j-i-r)}\right.\\
		&\left.\cdot\binom{l_{k+1}-j}{r}_p\binom{j-i}{r}_p\binom{d_{k}-l_{k+1}-j}{m-j-i}_p\prod_{e=0}^{r-1}(p^r-p^e)\right)
	\end{align*}
	is also uniformly given by a polynomial in $p$.
\end{proof}

\begin{thm}\label{thm:2step.graded}
	For $1\leq k\leq c-2$, let
	\[\alpha_1(a,k,p)=\#\{I\ideal_{k,1} L\mid \dim(I_k)+\dim(I_{k+1})=a\}.\]
Then we have
	\[\alpha_1(a,k,p)=\sum_{m=0}^{d_k}\sum_{i=0}^{\min(m,l_{k+1})}f_k(m,i,p)\binom{d_{k+1}-(2m-i)}{a-m-(2m-i)}_{p}.\]
	In particular, $\alpha_1(a,k,p)$ is uniformly given by a polynomial in $p$.
\end{thm}

\begin{proof}
	First, for each $0\leq m\leq d_{k}$ and $0\leq i \leq\min(m,l_{k+1})$, $f_k(m,i,p)$ gives the number of $I_{k}\leq L_{k}$ such that $\dim(I_{k})=m$ and $\dim(\phi_{1}(I_k)+\phi_{2}(I_{k}))=2m-i$. For such fixed $I_k$, to have $I_{k}\oplus I_{k+1}\triangleleft L_{k}\oplus L_{k+1}$ we need $I_{k+1}$ to contain $\phi(I_{k})$. Hence we need to count the number of $a-m$-dimensional subspace $I_{k+1}$ which contains a $2m-i$-dimensional fixed subspace $\phi(I_{k})$ in a $d_{k+1}$-dimensional vector space $L_{k+1}$, and this is $\binom{d_{k+1}-(2m-i)}{a-m-(2m-i)}_{p}$. Therefore we get
	\[\alpha_1(a,k,p)=\sum_{m=0}^{d_k}\sum_{i=0}^{\min(m,l_{k+1})}f_k(m,i,p)\binom{d_{k+1}-(2m-i)}{a-m-(2m-i)}_{p}.\]
\end{proof}
\begin{exm}\label{exm:k=2}
	Let us compare Theorem \ref{thm:dim.col2.poly} and Theorem \ref{thm:2step.graded} with Theorem \ref{thm:cleq5}.
	
	Suppose we compute 
	\[\alpha_1(a,2,p)=\#\{(I_{2},I_{3})\mid I_{2}\oplus I_{3}\triangleleft L_{2}\oplus L_{3},\dim(I_2)+\dim(I_{3})=a\}.\]
	Then by Theorem \ref{thm:2step.graded} we have
	\[\alpha_1(a,2,p)=f_2(0,0,p)\binom{2}{a}_p+f_2(1,0,p)\binom{0}{a-3}_p,\] which gives
	\begin{align*}
		\alpha_1(0,2,p)&=f_2(0,0,p)=1,&\alpha_1(1,2,p)&=f_2(0,0,p)\binom{2}{1}_p=\binom{2}{1}_p\\
		\alpha_1(2,2,p)&=f_2(0,0,p)=1,&\alpha_1(3,2,p)&=f_2(1,0,p)=1,
	\end{align*}
	and $\alpha_1(a,2,p)=0$ for $a\geq4$.
	One can directly check this by comparing them with Theorem \ref{thm:cleq5}, since by definition we must have
	\begin{align*}
		\alpha_1(0,2,p)&=b_{0,0,0}(\mff_{3,2}(\Fp)),&\alpha_1(1,2,p)&=b_{0,0,1}(\mff_{3,2}(\Fp)),\\
		\alpha_1(2,2,p)&=b_{0,0,2}(\mff_{3,2}(\Fp)),&\alpha_1(3,2,p)&=b_{0,1,2}(\mff_{3,2}(\Fp)),	
	\end{align*}
	and it is impossible to have $\dim(I_2)+\dim(I_{3})\geq4$.
	
	Similarly, suppose we compute $\alpha_1(a,3,p)$. By Theorem \ref{thm:2step.graded} we have
	\begin{align*}
		\alpha_1(a,3,p)=&f_3(0,0,p)\binom{3}{a}_p+f_3(1,0,p)\binom{1}{a-3}_p+f_3(1,1,p)\binom{2}{a-2}_p\\
		&+f_3(2,0,p)\binom{-1}{a-6}_p+f_3(2,1,p)\binom{0}{a-5}_p.
	\end{align*}
	With Theorem \ref{thm:dim.col2.poly} one can compute that
	\begin{align*}
		f_3(0,0,p)&=1,&f_3(1,0,p)&=p+1,&f_3(1,1,p)&=0,\\
		f_3(2,0,p)&=0,&f_3(2,1,p)&=1,&&	
	\end{align*}
	and \begin{align*}
		\alpha_1(0,3,p)&=f_3(0,0,p)=1,&\alpha_1(1,3,p)&=f_3(0,0,p)\binom{3}{1}_p=\binom{3}{1}_p\\
		\alpha_1(2,3,p)&=f_3(0,0,p)\binom{3}{2}_p=\binom{3}{2}_p,&\alpha_1(3,3,p)&=f_3(0,0,p)+f_3(1,0,p)=p+2,\\
		\alpha_1(4,3,p)&=f_3(1,0,p)=p+1,&\alpha_1(5,3,p)&=f_3(2,1,p)=1,
	\end{align*}
	and $\alpha_1(a,3,p)=0$ for $a\geq6$. Again, one can directly check this by comparing them with Theorem \ref{thm:cleq5}, since by definition we must have
	\begin{align*}
		\alpha_1(0,3,p)&=b_{0,0,0,0}(\mff_{3,3}(\Fp)),&\alpha_1(1,3,p)&=b_{0,0,0,1}(\mff_{3,3}(\Fp)),\\
		\alpha_1(2,3,p)&=b_{0,0,0,2}(\mff_{3,3}(\Fp)),&\alpha_1(3,3,p)&=b_{0,0,0,3}(\mff_{3,3}(\Fp))+b_{0,0,1,2}(\mff_{3,3}(\Fp)),\\
		\alpha_1(4,3,p)&=b_{0,0,1,3}(\mff_{3,2}(\Fp)),&\alpha_1(5,3,p)&=b_{0,0,2,3}(\mff_{3,2}(\Fp)),	
	\end{align*}
	and it is impossible to have $\dim(I_{3})+\dim(I_{4})\geq6$.
\end{exm}
\begin{cor}\label{thm:2sg.zeta}
	For $1\leq k \leq c-1$, the zeta function counting one-step graded ideals
	\[\zeta_{\mff_{c,2}(\Fp)}^{\ideal_1}(s)=\sum_{k=1}^{c-1}\sum_{a=0}^{d_{k}+d_{k+1}}\alpha_1(a,k,p)t^{d_{k}+d_{k+1}-a}\]
	is uniformly given by a polynomial in $p$. In other words, $\zeta_{\mff_{c,2}}^{\ideal_1}(s)$ is $\Fp$-uniform.
\end{cor}
\begin{proof}
	First, note that when $k=1$, one can directly check that 
	\begin{align*}
		\alpha_1(0,1,p)&=1&\alpha_1(1,1,p)&=1&\alpha_1(2,1,p)&=p+1&\alpha_1(3,1,p)&=1
	\end{align*}
	and $\alpha_1(a,1,p)=0$ for $a>3$, since $\zeta_{\mff_{2,2}(\Fp)}^{\idealgr}(s)=1+\binom{2}{1}_{p}t+t^2+t^3=\sum_{a=0}^{\infty}\alpha_1(a,1,p)t^{3-a}$ (c.f. Theorem \ref{thm:cleq5}). The rest follows from Theorem \ref{thm:2step.graded}. In particular, this proves Theorem \ref{thm:1step.uniform}.
\end{proof}

One might hope to get a similar result for some 2-step, or  general $n$-step graded ideals. It turns out that for some good cases, we can also compute the 2-step graded zeta functions.
\begin{thm}\label{thm:3step.graded}
	Let
	\[\alpha_2(a,k,p)=\#\{I\ideal_{k,2} L\mid \dim(I_k)+\dim(I_{k+1})+\dim(I_{k+2})=a\}.\] 
	Then, for $k=1,3,5$ and $c\geq k+2$ we have
	\[\alpha_2(a,k,p)=\sum_{m=0}^{d_k}\sum_{i=0}^{l_{k+1}}f_k(m,i,p)\sum_{m'=2m-i}^{d_{k+1}}\binom{d_{k+1}-(2m-i)}{m'}_{p}\binom{d_{k+2}-2m'}{a-m-m'}_{p}.\]
	In particular, for those $k$, $\alpha_2(a,k,p)$ is uniformly given by a polynomial in $p$.
\end{thm}
\begin{proof}
	First, note that as before,  for each $0\leq m\leq d_{k}$ and $0\leq i \leq l_{k+1}$, $f_k(m,i,p)$ gives the number of $I_{k}\leq L_{k}$ such that $\dim(I_{k})=m$ and $\dim(\phi_{1}(I_k)+\phi_{2}(I_{k}))=2m-i$. For such fixed $I_k$, to have $I_{k}\oplus I_{k+1}\oplus I_{k+2}\triangleleft L_{k}\oplus L_{k+1}\oplus L_{k+2}$ we need $I_{k+1}$ to contain $\phi(I_{k}):=\phi_{1}(I_{k})+\phi_{2}(I_{k})$. For each $m'$ where $\dim(\phi(I_{k}))=2m-i\leq m'\leq d_{k+1}$, there are exactly $\binom{d_{k+1}-(2m-i)}{m'}_{p}$ such $I_{k+1}$. 
	
	Now, the crucial part is that for $k=1,3,5$, $\Lambda_{k+2}=0$, which implies $\dim(\phi(I_{k+1}))=2\dim(I_{k+1})=2m'$. So for any fixed $m$ and $m'$, there are exactly $\binom{d_{k+2}-2m'}{a-m-m'}_{p}$ number of $a-m-m'$-dimensional subspace $I_{k+2}\leq L_{k+2}$ such that $I_{k+2}$ contains $2m'$-dimensional vector space $\phi(I_{k+1})$ and $\dim(I_k)+\dim(I_{k+1})+\dim(I_{k+2})=m+m'+\dim(I_{k+2})=a$.
\end{proof}
Furthermore, note that by looking at Theorem \ref{thm:cleq5} one can also check that $\alpha_2(a,2,p)$ is uniformly given by a polynomial in $p$ for $c\geq4$.

Interestingly, it turns out that $\alpha_2(a,4,p)$  is much more complicated than the ones we just computed.  In fact, this is where the $\Fp$-uniformity fails for $\mff_{6,2}$.

\begin{rem}
	Although here we specifically concentrated on the case where $L=\mff_{c,d}(\Fp)$, one can of course consider the general $n$-step graded ideal zeta function of $L$ defined in Definition \ref{def:nsg} as	
	\begin{align*}
		\zeta_{L}^{\idealn}(s):=\sum_{k=1}^{c-n}\sum_{I\idealkn L}|L_{k}:I_{k}|^{-s}\cdots|L_{k+n}:I_{k+n}|^{-s}.
	\end{align*}
In this most general setting $L$ does not even need to be an $\Fp$-Lie algebra. Can we also prove the analogous result that $\zeta_{\mff_{c,2}}^{\ideal_1}(s)$ is $\Zp$-uniform for all $c$? Or can we find other interesting Lie rings $L$ such that their $n$-step graded ideal is known to be $\Zp$-uniform or $\Fp$-uniform for any $n$? More examples would be useful to initiate a general study.
\end{rem}

\section{Counting graded ideals of $\mff_{c,2}(\Fp)$ for $c=6$.}\label{sec:c=6}

In this section, we prove that $\zeta_{\mff_{6,2}}^{\idealgr}(s)$ is not $\Fp$-uniform.

\begin{thm}\label{thm:c=6}
Let	
\[\zeta_{\mff_{6,2}(\Fp)}^{\idealgr}(s)=\sum_{i=0}^{23}a_{p^{i}}^{\idealgr}(\mff_{6,2}(\Fp))t^{i}.\]
For $p\geq5$, we have
\begin{align*}
	a_{p^{9}}^{\idealgr}(\mff_{6,2}(\Fp))=\begin{cases}
		p^6+3p^5+5p^4+6p^3+9p^2+10p+4&p\equiv3,5\mod8,\\
		p^6+3p^5+5p^4+6p^3+9p^2+8p+4&p\equiv1,7\mod8,
	\end{cases}
\end{align*}
In particular, $\zeta_{\mff_{6,2}(\Fp)}^{\idealgr}(s)$ is not $\Fp$-uniform.
\end{thm}

Again, throughout this section we write $L=\mff_{6,2}(\Fp)$. To prove this, we need to choose a basis for $L$.  In Table~\ref{tab:basis} we record the set of basis we chose for $L=\mff_{6,2}$ with their relations. This particular set of basis is achieved by GAP \cite{GAP4} using LieRing package \cite{LieRing242}.
\begin{table}\label{tab:basis} 
	\centering 
	\caption{A basis of $\mff_{c,2}$ for $c\leq6$}
	\begin{tabular}{|c|c|c|c|} \hline
		$k$& Basis for $L_{k}$ & Notation & Jacobi Relation   \\ \hline
		$k=1$&$X_1$ & $X_{1,1}$ &    \\
		&$X_2$ & $X_{1,2}$ &    \\ \hline
		$k=2$&$[X_1,X_2]$ & $X_{2,1}$ &      \\ \hline
		$k=3$&$[X_1,[X_1,X_2]]$ & $X_{3,1}$ & $\phi_{1}(X_{2,1})$    \\
		&$[X_2,[X_1,X_2]]$ & $X_{3,2}$ & $\phi_{2}(X_{2,1})$    \\ \hline
		$k=4$&$[X_1,[X_1,[X_1,X_2]]]$ & $X_{4,1}$ & $\phi_{1}(X_{3,1})$    \\
		&$[X_2,[X_1,[X_1,X_2]]]$ & $X_{4,2}$ & $\phi_{2}(X_{3,1})=\phi_{1}(X_{3,2})$    \\
		&$[X_2,[X_2,[X_1,X_2]]]$ & $X_{4,3}$ & $\phi_{2}(X_{3,2})$    \\ \hline
		$k=5$&$[X_1,[X_1,[X_1,[X_1,X_2]]]]$ & $X_{5,1}$ & $\phi_{1}(X_{4,1})$    \\
		&$[X_1,[X_2,[X_1,[X_1,X_2]]]]$ & $X_{5,2}$ & $\phi_{1}(X_{4,2})$   \\
		&$[X_1,[X_2,[X_2,[X_1,X_2]]]]$ & $X_{5,3}$ & $\phi_{1}(X_{4,3})$   \\
		&$[X_2,[X_1,[X_1,[X_1,X_2]]]]$ & $X_{5,4}$ & $\phi_{2}(X_{4,1})$   \\ 
		&$[X_2,[X_2,[X_1,[X_1,X_2]]]]$ & $X_{5,5}$ & $\phi_{2}(X_{4,2})$   \\
		&$[X_2,[X_2,[X_2,[X_1,X_2]]]]$ & $X_{5,6}$ & $\phi_{2}(X_{4,3})$   \\ \hline
		$k=6$&$[X_1,[X_1,[X_1,[X_1,[X_1,X_2]]]]]$ & $X_{6,1}$ & $\phi_{1}(X_{5,1})$   \\
		&$[X_1,[X_1,[X_2,[X_1,[X_1,X_2]]]]]$ & $X_{6,2}$ & $\phi_{1}(X_{5,2})=2X_{6,4}-\phi_2(X_{5,1})$    \\
		&$[X_1,[X_1,[X_2,[X_2,[X_1,X_2]]]]]$ & $X_{6,3}$ & $\phi_{1}(X_{5,3})=\phi_2(X_{5,4})+3X_{6,5}-3X_{6,7}$    \\
		&$[X_1,[X_2,[X_1,[X_1,[X_1,X_2]]]]]$ & $X_{6,4}$ & $\phi_{1}(X_{5,4})$    \\
		&$[X_1,[X_2,[X_2,[X_1,[X_1,X_2]]]]]$ & $X_{6,5}$ & $\phi_{1}(X_{5,5})$    \\ 
		&$[X_1,[X_2,[X_2,[X_2,[X_1,X_2]]]]]$ & $X_{6,6}$ & $\phi_{1}(X_{5,6})=2X_{6,8}-\phi_2(X_{5,5})$   \\
		&$[X_2,[X_1,[X_2,[X_1,[X_1,X_2]]]]]$ & $X_{6,7}$ & $\phi_{2}(X_{5,2})$    \\
		&$[X_2,[X_1,[X_2,[X_2,[X_1,X_2]]]]]$ & $X_{6,8}$ & $\phi_{2}(X_{5,3})$    \\ 
		&$[X_2,[X_2,[X_2,[X_2,[X_1,X_2]]]]]$ & $X_{6,9}$ & $\phi_{2}(X_{5,6})$    \\ \hline
	\end{tabular}
\end{table}
Note that $L_1=\langle X_{1,1},X_{1,2}\rangle$, $L_2=\langle X_{2,1}\rangle$, $L_3=\langle X_{3,1},X_{3,2}\rangle$, $L_4=\langle X_{4,1},X_{4,2}, X_{4,3}\rangle$, $L_5=\langle X_{5,1},\ldots, X_{5,6}\rangle$, and $L_6=\langle X_{6,1},\ldots, X_{6,9}\rangle$.

We prove Theorem \ref{thm:c=6} in several steps by considering individual cases.

\begin{proof}
To compute $a_{p^{9}}^{\idealgr}(L)$, we need to count all the graded ideals of dimension $23-9=14$ in $L$. Write
\[b_{\boldsymbol{m}}(L)=\{I\idealgr L\mid \dim(I_1)=m_1,\ldots,\dim(I_6)=m_6\}.\]
By the same argument used in Theorem \ref{thm:cleq5}, the potential candidates contributing to $a_{p^{9}}^{\idealgr}(L)$ are
\begin{align*}
	b_{0,0,0,2,4,8}(L),&&b_{0,0,0,2,5,7}(L),&&b_{0,0,0,1,4,9}(L),\\
	b_{0,0,0,1,5,8}(L),&&b_{0,0,0,0,5,9}(L).&&	
\end{align*}
First, note that there are $\binom{3}{1}_p=p^2+p+1$ 1-dimensional subspaces $I_4$ in $L_4$, and any such $I_4$ would induce a 2-dimensional $\phi(I_4)$ in a 6-dimensional vector space $L_5$. Hence for each fixed 1-dimensional $I_4$, the number of 4-dimensional subspaces $I_5$ in $L_5$ that contains $\phi(I_4)$ is $\binom{6-2}{4-2}_p=\binom{4}{2}_p=p^4+p^3+2p^2+p+1$. Since there is only one 9-dimensional subspace in $L_6$, namely $L_6$ itself, and it clearly contains $\phi(I_5)$ for any 4-dimensional $I_5$, we have
\[b_{0,0,0,1,4,9}(L)=\binom{3}{1}_p\binom{4}{2}_p=p^6+2p^5+4p^4+4p^3+4p^2+2p+1.\]

By the same argument, one can also easily see that
\[b_{0,0,0,0,5,9}(L)=\binom{6}{5}_p=p^5+p^4+p^3+p^2+p^1+1.\]

Hence we only need to compute $b_{0,0,0,2,4,8}(L)$, $b_{0,0,0,2,5,7}(L)$, and $b_{0,0,0,1,5,8}(L)$. We start with $b_{0,0,0,2,4,8}(L)$ and $b_{0,0,0,2,5,7}(L)$.

\subsection{Case 1: $\dim(I_4)=2$}
Let us fix a 2-dimensional subspace $I_4$ of $L_4=\langle X_{4,1},X_{4,2},X_{4,3} \rangle$. In this case we have
\[I_4=\begin{pmatrix}
	1&0&a_3\\
	0&1&b_3
\end{pmatrix},
\begin{pmatrix}
	1&a_2&0\\
	0&0&1
\end{pmatrix},
\begin{pmatrix}
	0&1&0\\
	0&0&1
\end{pmatrix}\]
where $a_2,a_3,b_3\in\Fp$, which matches with the fact that there are $\binom{3}{2}_p=p^2+p+1$ 2-dimensional $I_4$ in $L_4$. We separate our cases accordingly.
\subsubsection{Case 1-1: $I_4=\begin{pmatrix}
		1&0&a_3\\
		0&1&b_3
	\end{pmatrix}$}\label{subsubsec:case3-1}

First, note that there are $p^2$ such $I_4$ in $L_4$. Recall Table \ref{tab:basis}. If we choose a general element  $a=a_1X_{4,1}+a_2X_{4,2}+a_3X_{4,3}\in L_4$, where $a\neq0$, then we have 
\begin{align*}
	\phi_{1}(a)&=a_1X_{5,1}+a_2X_{5,2}+a_3X_{5,3},\\
	\phi_{2}(a)&=a_1X_{5,4}+a_2X_{5,5}+a_3X_{5,6}, 
\end{align*}
and
\begin{align*}
	\phi_{1}\phi_{1}(a)&=a_1X_{6,1}+a_2X_{6,2}+a_3X_{6,3},\\
	\phi_{1}\phi_{2}(a)&=a_1X_{6,4}+a_2X_{6,5}+a_3X_{6,6},\\
	\phi_{2}\phi_{1}(a)&=-a_1X_{6,2}+2a_1X_{6,4}+a_2X_{6,7}+a_3X_{6,8},\\
	\phi_{2}\phi_{2}(a)&=a_1X_{6,3}-3a_1X_{6,5}-a_2X_{6,6}+(3a_1+2a_2)X_{6,8}+a_3X_{6,9}. 
\end{align*}
Therefore we have
\begin{align*}\phi(I_4)&=\begin{pmatrix}
		1 & 0 & a_3 & 0 & 0 & 0 \\
		0 & 1 & b_3 & 0 & 0 & 0 \\
		0 & 0 & 0 & 1 & 0 & a_3 \\
		0 & 0 & 0 & 0 & 1 & b_3 
	\end{pmatrix}
\end{align*}
and 
\begin{align*}\phi(\phi(I_4))&=\begin{pmatrix}
		1 & 0 & a_3 & 0 & 0 & 0 & 0 & 0 & 0 \\
		0 & 1 & b_3 & 0 & 0 & 0 & 0 & 0 & 0 \\
		0 & 0 & 0 & 1 & 0 & a_3 & 0 & 0 & 0 \\
		0 & 0 & 0 & 0 & 1 & b_3 & 0 & 0 & 0 \\
		0 & -1 & 0 & 2 & 0 & 0 & 0 & a_3 & 0 \\
		0 & 0 & 0 & 0 & 0 & 0 & 1 & b_3 & 0 \\
		0 & 0 & 1 & 0 & -3 & 0 & 0 & 3 & a_3 \\
		0 & 0 & 0 & 0 & 0 & -1 & 0 & 2 & b_3 
	\end{pmatrix}.
\end{align*}
One can directly check that here $\phi(I_4)$ always has rank 4 for any choice of $a_3,b_3\in\Fp$, and $\phi(\phi(I_4))$ has
\begin{itemize}
\item rank 7 if $a_3=b_3=0$ or $a_3=b_3=-1$, and
\item rank 8 otherwise.
\end{itemize}
Let $b_{\bfm}^{1,1}(L)$ denote the number of graded ideal $I\idealgr L$ of dimension $\bfm$ where 
\[I_4=\begin{pmatrix}
	1&0&a_3\\
	0&1&b_3
\end{pmatrix}.\] Our computation shows that 
\begin{align*}
	b_{0,0,0,2,4,7}^{1,1}&=2,\\
	b_{0,0,0,2,4,8}^{1,1}&=(p^2-2)+b_{0,0,0,2,4,7}^{1,1}\cdot\binom{9-7}{8-7}_p=p^2+2p.
\end{align*}

Now suppose $\dim(I_5)=5$. Since $\phi(I_4)\leq I_5$, we have
\begin{align*}
	I_5=\begin{pmatrix}
		1 & 0 & a_3 & 0 & 0 & 0 \\
		0 & 1 & b_3 & 0 & 0 & 0 \\
		0 & 0 & 0 & 1 & 0 & a_3 \\
		0 & 0 & 0 & 0 & 1 & b_3 \\
		0 & 0 & 1 & 0 & 0 & c_6 
	\end{pmatrix}\textrm{ or }
	\begin{pmatrix}
		1 & 0 & a_3 & 0 & 0 & 0 \\
		0 & 1 & b_3 & 0 & 0 & 0 \\
		0 & 0 & 0 & 1 & 0 & 0 \\
		0 & 0 & 0 & 0 & 1 & 0 \\
		0 & 0 & 0 & 0 & 0 & 1 
	\end{pmatrix},
\end{align*}
giving
 \begin{align*}
 	\phi(I_5)=\begin{pmatrix}
 			1 & 0 & a_3 & 0 & 0 & 0 & 0 & 0 & 0 \\
 			0 & 1 & b_3 & 0 & 0 & 0 & 0 & 0 & 0 \\
 			0 & 0 & 0 & 1 & 0 & a_3 & 0 & 0 & 0 \\
 			0 & 0 & 0 & 0 & 1 & b_3 & 0 & 0 & 0 \\
 			0 & 0 & 1 & 0 & 0 & c_6 & 0 & 0 & 0 \\
 			0 & -1 & 0 & 2 & 0 & 0 & 0 & a_3 & 0 \\
 			0 & 0 & 0 & 0 & 0 & 0 & 1 & b_3 & 0 \\
 			0 & 0 & 1 & 0 & -3 & 0 & 0 & 3 & a_3 \\
 			0 & 0 & 0 & 0 & 0 & -1 & 0 & 2 & b_3 \\
 			0 & 0 & 0 & 0 & 0 & 0 & 0 & 1 & c_6 
 	\end{pmatrix}\textrm{ or }
 	\begin{pmatrix}
 			1 & 0 & a_3 & 0 & 0 & 0 & 0 & 0 & 0 \\
 			0 & 1 & b_3 & 0 & 0 & 0 & 0 & 0 & 0 \\
 			0 & 0 & 0 & 1 & 0 & 0 & 0 & 0 & 0 \\
 			0 & 0 & 0 & 0 & 1 & 0 & 0 & 0 & 0 \\
 			0 & 0 & 0 & 0 & 0 & 1 & 0 & 0 & 0 \\
 			0 & -1 & 0 & 2 & 0 & 0 & 0 & a_3 & 0 \\
 			0 & 0 & 0 & 0 & 0 & 0 & 1 & b_3 & 0 \\
 			0 & 0 & 1 & 0 & -3 & 0 & 0 & 3 & 0 \\
 			0 & 0 & 0 & 0 & 0 & -1 & 0 & 2 & 0 \\
 			0 & 0 & 0 & 0 & 0 & 0 & 0 & 0 & 1 
 	\end{pmatrix},
 \end{align*}
where $c_6\in\Fp$. One can check that none of $\phi(I_5)$ has rank $7$. Hence $b_{0,0,0,2,5,7}^{1,1}(L)=0$.

\subsubsection{Case 1-2: $I_4=\begin{pmatrix}
		1&a_2&0\\
		0&0&1
	\end{pmatrix}$}\label{subsubsec:case3-2}

First, note that there are $p$ such $I_4$ in $L_4$. Let $b_{\bfm}^{1,2}(L)$ denote the number of graded ideal $I\idealgr L$ of dimension $\bfm$ where \[I_4=\begin{pmatrix}
	1&a2&0\\
	0&0&1
\end{pmatrix}.\] Since
\begin{align*}\phi(I_4)&=\begin{pmatrix}
		1 & a_2 & 0 & 0 & 0 & 0 \\
		0 & 0 & 1 & 0 & 0 & 0 \\
		0 & 0 & 0 & 1 & a_2 & 0 \\
		0 & 0 & 0 & 0 & 0 & 1 
	\end{pmatrix}
\end{align*}
always has rank 4, and 
\begin{align*}\phi(\phi(I_4))&=\begin{pmatrix}
		1 & a_2 & 0 & 0 & 0 & 0 & 0 & 0 & 0 \\
		0 & 0 & 1 & 0 & 0 & 0 & 0 & 0 & 0 \\
		0 & 0 & 0 & 1 & a_2 & 0 & 0 & 0 & 0 \\
		0 & 0 & 0 & 0 & 0 & 1 & 0 & 0 & 0 \\
		0 & -1 & 0 & 2 & 0 & 0 & a_2 & 0 & 0 \\
		0 & 0 & 0 & 0 & 0 & 0 & 0 & 1 & 0 \\
		0 & 0 & 1 & 0 & -3 & -a_2 & 0 & 2 a_2+3 & 0 \\
		0 & 0 & 0 & 0 & 0 & 0 & 0 & 0 & 1 
	\end{pmatrix}
\end{align*}
always has rank 8,  we have
\begin{align*}
	b_{0,0,0,2,4,8}^{1,2}(L)=p.
\end{align*}

Now suppose $\dim(I_5)=5$. We have
\begin{align*}
	I_5=\begin{pmatrix}
		1 & a_2 & 0 & 0 & 0 & 0 \\
		0 & 0 & 1 & 0 & 0 & 0 \\
		0 & 0 & 0 & 1 & a_2 & 0 \\
		0 & 0 & 0 & 0 & 0 & 1 \\
		0 & 1 & 0 & 0 & c_5 & 0 
	\end{pmatrix}\textrm{ or }
	\begin{pmatrix}
		1 & a_2 & 0 & 0 & 0 & 0 \\
		0 & 0 & 1 & 0 & 0 & 0 \\
		0 & 0 & 0 & 1 & 0 & 0 \\
		0 & 0 & 0 & 0 & 0 & 1 \\
		0 & 0 & 0 & 0 & 1 & 0 
	\end{pmatrix},
\end{align*} giving
\begin{align*}
\phi(I_5)=\begin{pmatrix}
		1 & a_2 & 0 & 0 & 0 & 0 & 0 & 0 & 0 \\
		0 & 0 & 1 & 0 & 0 & 0 & 0 & 0 & 0 \\
		0 & 0 & 0 & 1 & a_2 & 0 & 0 & 0 & 0 \\
		0 & 0 & 0 & 0 & 0 & 1 & 0 & 0 & 0 \\
		0 & 1 & 0 & 0 & c_5 & 0 & 0 & 0 & 0 \\
		0 & -1 & 0 & 2 & 0 & 0 & a_2 & 0 & 0 \\
		0 & 0 & 0 & 0 & 0 & 0 & 0 & 1 & 0 \\
		0 & 0 & 1 & 0 & -3 & -a_2 & 0 & 2 a_2+3 & 0 \\
		0 & 0 & 0 & 0 & 0 & 0 & 0 & 0 & 1 \\
		0 & 0 & 0 & 0 & 0 & -c_5 & 1 & 2 c_5 & 0 
\end{pmatrix}\textrm{ or }
\begin{pmatrix}
		1 & a_2 & 0 & 0 & 0 & 0 & 0 & 0 & 0 \\
		0 & 0 & 1 & 0 & 0 & 0 & 0 & 0 & 0 \\
		0 & 0 & 0 & 1 & a_2 & 0 & 0 & 0 & 0 \\
		0 & 0 & 0 & 0 & 0 & 1 & 0 & 0 & 0 \\
		0 & 0 & 0 & 0 & 1 & 0 & 0 & 0 & 0 \\
		0 & -1 & 0 & 2 & 0 & 0 & a_2 & 0 & 0 \\
		0 & 0 & 0 & 0 & 0 & 0 & 0 & 1 & 0 \\
		0 & 0 & 1 & 0 & -3 & -a_2 & 0 & 2 a_2+3 & 0 \\
		0 & 0 & 0 & 0 & 0 & 0 & 0 & 0 & 1 \\
		0 & 0 & 0 & 0 & 0 & -1 & 0 & 2 & 0 
\end{pmatrix},
\end{align*}
where $c_5\in\Fp$. Again none of $\phi(I_5)$ has rank $7$. Hence $b_{0,0,0,2,5,7}^{1,2}(L)=0$.

\subsubsection{Case 1-3: $I_4=\begin{pmatrix}
		0&1&0\\
		0&0&1
	\end{pmatrix}$}\label{subsubsec:case3-3}

First, note that there are exactly 1 such $I_4$ in $L_4$. Let $b_{\bfm}^{1,3}(L)$ denote the number of graded ideal $I\idealgr L$ of dimension $\bfm$ where \[I_4=\begin{pmatrix}
	0&1&0\\
	0&0&1
\end{pmatrix}.\] Since 
\begin{align*}\phi(I_4)&=\begin{pmatrix}
		0 & 1 & 0 & 0 & 0 & 0 \\
		0 & 0 & 1 & 0 & 0 & 0 \\
		0 & 0 & 0 & 0 & 1 & 0 \\
		0 & 0 & 0 & 0 & 0 & 1 
	\end{pmatrix}
\end{align*}
always has rank 4, and 
\begin{align*}\phi(\phi(I_4))&=\begin{pmatrix}
		0 & 1 & 0 & 0 & 0 & 0 & 0 & 0 & 0 \\
		0 & 0 & 1 & 0 & 0 & 0 & 0 & 0 & 0 \\
		0 & 0 & 0 & 0 & 1 & 0 & 0 & 0 & 0 \\
		0 & 0 & 0 & 0 & 0 & 1 & 0 & 0 & 0 \\
		0 & 0 & 0 & 0 & 0 & 0 & 1 & 0 & 0 \\
		0 & 0 & 0 & 0 & 0 & 0 & 0 & 1 & 0 \\
		0 & 0 & 0 & 0 & 0 & -1 & 0 & 2 & 0 \\
		0 & 0 & 0 & 0 & 0 & 0 & 0 & 0 & 1 
	\end{pmatrix}
\end{align*}
always has rank 7,  we have
\begin{align*}
	b_{0,0,0,2,4,7}^{1,3}(L)&=1,\\
	b_{0,0,0,2,4,8}^{3,3}(L)&=b_{0,0,0,2,4,7}^{1,3}(L)\cdot\binom{2}{1}_p=p+1.
\end{align*}

Now suppose $\dim(I_5)=5$. We have
\begin{align*}
	I_5=\begin{pmatrix}
		0 & 1 & 0 & 0 & 0 & 0 \\
		0 & 0 & 1 & 0 & 0 & 0 \\
		0 & 0 & 0 & 0 & 1 & 0 \\
		0 & 0 & 0 & 0 & 0 & 1 \\
		1 & 0 & 0 & c_4 & 0 & 0 
	\end{pmatrix}\textrm{ or }
	\begin{pmatrix}
		0 & 1 & 0 & 0 & 0 & 0 \\
		0 & 0 & 1 & 0 & 0 & 0 \\
		0 & 0 & 0 & 0 & 1 & 0 \\
		0 & 0 & 0 & 0 & 0 & 1 \\
		0 & 0 & 0 & 1 & 0 & 0 
	\end{pmatrix},
\end{align*}giving
\begin{align*}
\phi(I_5)=\begin{pmatrix}
	0 & 1 & 0 & 0 & 0 & 0 & 0 & 0 & 0 \\
	0 & 0 & 1 & 0 & 0 & 0 & 0 & 0 & 0 \\
	0 & 0 & 0 & 0 & 1 & 0 & 0 & 0 & 0 \\
	0 & 0 & 0 & 0 & 0 & 1 & 0 & 0 & 0 \\
	1 & 0 & 0 & c_4 & 0 & 0 & 0 & 0 & 0 \\
	0 & 0 & 0 & 0 & 0 & 0 & 1 & 0 & 0 \\
	0 & 0 & 0 & 0 & 0 & 0 & 0 & 1 & 0 \\
	0 & 0 & 0 & 0 & 0 & -1 & 0 & 2 & 0 \\
	0 & 0 & 0 & 0 & 0 & 0 & 0 & 0 & 1 \\
	0 & -1 & c_4 & 2 & -3 c_4 & 0 & 0 & 3 c_4 & 0 
\end{pmatrix}\textrm{ or }
\begin{pmatrix}
		0 & 1 & 0 & 0 & 0 & 0 & 0 & 0 & 0 \\
		0 & 0 & 1 & 0 & 0 & 0 & 0 & 0 & 0 \\
		0 & 0 & 0 & 0 & 1 & 0 & 0 & 0 & 0 \\
		0 & 0 & 0 & 0 & 0 & 1 & 0 & 0 & 0 \\
		0 & 0 & 0 & 1 & 0 & 0 & 0 & 0 & 0 \\
		0 & 0 & 0 & 0 & 0 & 0 & 1 & 0 & 0 \\
		0 & 0 & 0 & 0 & 0 & 0 & 0 & 1 & 0 \\
		0 & 0 & 0 & 0 & 0 & -1 & 0 & 2 & 0 \\
		0 & 0 & 0 & 0 & 0 & 0 & 0 & 0 & 1 \\
		0 & 0 & 1 & 0 & -3 & 0 & 0 & 3 & 0 
\end{pmatrix},
\end{align*}
where $c_4\in\Fp$. None of $\phi(I_5)$ has rank $7$. Hence $b_{0,0,0,2,5,7}^{1,3}(L)=0$.

To conclude $\dim(I_4)=2$ case, we have
\begin{align*}
	b_{0,0,0,2,4,8}(L)&=p^2+4p+1,&b_{0,0,0,2,5,7}(L)&=0.
\end{align*}

\subsection{Case 2: $\dim(I_4)=1$}
Now let us consider the case where $\dim(I_4)=1$. We need to compute $b_{0,0,0,1,5,8}(L)$. 

Let us fix a 1-dimensional subspace $I_4$ of $L_4$. Here we have
\[I_4=\begin{pmatrix}
	1&a_2&a_3
\end{pmatrix},
\begin{pmatrix}
	0&1&a_3
\end{pmatrix},
\begin{pmatrix}
	0&0&1
\end{pmatrix}\]
where $a_2,a_3\in\Fp$, which matches with the fact that there are $\binom{3}{1}_p=p^2+p+1$ 1-dimensional $I_4$ in $L_4$. As before, we separate our cases accordingly.

\subsubsection{Case 2-1: $I_4=\begin{pmatrix}
		1&a_2&a_3
	\end{pmatrix}$}\label{subsubsec:case2-1}

 Let $b_{\bfm}^{2,1}(L)$ denote the number of graded ideal $I\idealgr L$ of dimension $\bfm$ where \[I_4=\begin{pmatrix}
 	1&a_2&a_3
 \end{pmatrix}.\] Note that there are $p^2$ such $I_4$ in $L_4$. This gives
\begin{align*}\phi(I_4)&=\begin{pmatrix}
		1&a_2&a_3&0&0&0\\
		0&0&0&1&a_2&a_3
	\end{pmatrix}.
\end{align*}

Since we need $\dim(I_5)=5$, there are four possible forms of $I_5$, namely
\begin{align*}\begin{pmatrix}
		1&a_2&a_3&0&0&0\\
		0&0&0&1&a_2&a_3\\
		0&1&0&0&0&b_6\\
		0&0&1&0&0&c_6\\
		0&0&0&0&1&e_6
	\end{pmatrix},&
	\begin{pmatrix}
		1&a_2&a_3&0&0&0\\
		0&0&0&1&a_2&0\\
		0&1&0&0&b_5&0\\
		0&0&1&0&c_5&0\\
		0&0&0&0&0&1
	\end{pmatrix},\\
	\begin{pmatrix}
		1&a_2&a_3&0&0&0\\
		0&0&0&1&0&0\\
		0&1&b_3&0&0&0\\
		0&0&0&0&1&0\\
		0&0&0&0&0&1
	\end{pmatrix},&
	\begin{pmatrix}
		1&a_2&0&0&0&0\\
		0&0&0&1&0&0\\
		0&0&1&0&0&0\\
		0&0&0&0&1&0\\
		0&0&0&0&0&1
	\end{pmatrix},
\end{align*}
where $b_3,b_5,b_6,c_5,c_6,e_6\in\Fp$. Note that it matches with the fact that there are $\binom{4}{3}_p=p^3+p^2+p+1$ 5-dimensional subspaces in $L_5$ containing $I_4$. 

Let \begin{align*}I_5=\begin{pmatrix}
		1&a_2&a_3&0&0&0\\
		0&0&0&1&a_2&a_3\\
		0&1&0&0&0&b_6\\
		0&0&1&0&0&c_6\\
		0&0&0&0&1&e_6
	\end{pmatrix}
\end{align*}
We have
\begin{align*}
	\phi(I_5)=\begin{pmatrix}
			1 & a_2 & a_3 & 0 & 0 & 0 & 0 & 0 & 0 \\
			0 & 0 & 0 & 1 & a_2 & a_3 & 0 & 0 & 0 \\
			0 & 1 & 0 & 0 & 0 & b_6 & 0 & 0 & 0 \\
			0 & 0 & 1 & 0 & 0 & c_6 & 0 & 0 & 0 \\
			0 & 0 & 0 & 0 & 1 & e_6 & 0 & 0 & 0 \\
			0 & -1 & 0 & 2 & 0 & 0 & a_2 & a_3 & 0 \\
			0 & 0 & 1 & 0 & -3 & -a_2 & 0 & 2 a_2+3 & a_3 \\
			0 & 0 & 0 & 0 & 0 & 0 & 1 & 0 & b_6 \\
			0 & 0 & 0 & 0 & 0 & 0 & 0 & 1 & c_6 \\
			0 & 0 & 0 & 0 & 0 & -1 & 0 & 2 & e_6 
	\end{pmatrix}.
\end{align*}
One can check that for $p\geq5$, this has rank 8 if and only if $a_2,a_3,b_6,c_6,e_6\in\Fp$ satisfies the system of equations
\begin{align*}
	0=&3 a_2 b_6 c_6-3 a_2 b_6 e_6+a_2^2b_6-a_3 b_6+4 a_2^2 c_6 e_6+6 a_2 c_6 e_6-3 a_3 c_6 e_6\\
	&-3 a_3 a_2 c_6+a_3 c_6^2-6 a_3 c_6-2 a_3 a_2 e_6+2 a_3^2+3 b_6 c_6,\\
	0=&2 a_2 b_6 c_6-4 a_2 b_6 e_6-2 a_3 b_6-3 a_2 b_6+a_3 c_6 e_6\\
	&-3 a_3 e_6^2-7 a_2 a_3 e_6-6 a_3 e_6+4 a_2^2 e_6^2+6 a_2 e_6^2+3 a_3^2+3 b_6 e_6,\\
	0=&7 a_2 c_6 e_6-2 a_2 c_6^2+3 a_2 c_6+a_2^2 e_6-3 a_2 e_6^2-a_2 a_3,\\
	0=&-a_2 b_6-4 a_2 c_6 e_6+3 a_3 c_6+2 a_2 e_6^2-2 a_3 e_6-2 b_6 c_6+b_6 e_6,\\
	0=&-a_2 e_6+a_3-7 c_6 e_6+2 c_6^2-3 c_6+3 e_6^2,\\
	0=&-3 a_2 b_6+2 a_2 c_6 e_6-4 a_2 c_6^2+6 a_2 c_6+9 a_3 c_6+2 a_2^2 e_6-6 a_3 e_6-2 a_3 a_2-6 b_6 c_6\\
	&+3 b_6 e_6.
\end{align*}
From the 5th equation, note that $a_3=a_2 d_6+7 c_6 d_6-2 c_6^2+3 c_6-3 d_6^2$. Substituting  this, luckily the rest of the equations become   
\[\left(-a_2-2 c_6+e_6\right)b_6 -a_2 c_6 e_6+25 c_6^2 e_6-23 c_6 e_6^2-6 c_6 e_6-6 c_6^3+9 c_6^2+6 e_6^3=0,\]
a linear equation in $b_6$. 

First, suppose $-a_2-2 c_6+e_6\neq0$. Then for such $a_2,c_6,e_6\in\Fp$, we simply get \[b_6=\frac{-a_2 c_6 e_6+25 c_6^2 e_6-23 c_6 e_6^2-6 c_6 e_6-6 c_6^3+9 c_6^2+6 e_6^3}{a_2+2 c_6-e_6}\] and \[a_3=a_2 d_6+7 c_6 d_6-2 c_6^2+3 c_6-3 d_6^2,\] giving $p^3-p^2$ cases.

Now, suppose $-a_2-2 c_6+e_6=0$. Putting $e_6=a_2+2c_6$, our equation now reduces into 
\begin{equation}\label{eq:residue}
12 a_2^2 c_6+3 a_2 c_6^2-6 a_2 c_6+6 a_2^3-3 c_6^2=0,
\end{equation}
where we can have a free variable $b_{6}\in\Fp$. Solving \eqref{eq:residue}, one get either 
\begin{enumerate}
\item $a_2=1$ and $c_6=-1$, giving 1 case, or
\item $a_2\neq1$ and $c_6=\frac{-2 a_2^2+a_2\pm a_2\sqrt{2 a_2^2-2 a_2+1}}{a_2-1}$.
\end{enumerate}
The latter gives
\begin{itemize}
	\item $p-2$ possible $(a_2,c_6)\in\Fp^{2}$, if $p\equiv 3,5\mod8$, and
	\item $p-4$ possible $(a_2,c_6)\in\Fp^{2}$, if $p\equiv 1,7\mod8$.
\end{itemize}
Hence for $-a_2-2 c_6+e_6=0$ we have
\begin{itemize}
	\item $p(p-2+1)=p^2-p$ possible $(a_2,a_3,b_6,c_6,e_6)\in\Fp^{5}$, if $p\equiv 3,5\mod8$, and
	\item $p(p-4+1)=p^2-3p$ possible $(a_2,a_3,b_6,c_6,e_6)\in\Fp^{5}$, if $p\equiv 1,7\mod8$.
\end{itemize}
Summing all up, one can see that
\begin{itemize}
	\item $p^3-p$ of 	$\phi(I_5)$ have rank 8, if $p\equiv3,5\mod8$, and
	\item  $p^3-3p$ of $\phi(I_5)$ have rank 8, if $p\equiv1,7\mod8$.
\end{itemize}
Similarly, let \begin{align*}I_5=\begin{pmatrix}
		1&a_2&a_3&0&0&0\\
		0&0&0&1&a_2&0\\
		0&1&0&0&b_5&0\\
		0&0&1&0&c_5&0\\
		0&0&0&0&0&1
	\end{pmatrix}.
\end{align*}
We have 
\begin{align*}
	\phi(I_5)=\begin{pmatrix}
			1 & a_2 & a_3 & 0 & 0 & 0 & 0 & 0 & 0 \\
			0 & 0 & 0 & 1 & a_2 & 0 & 0 & 0 & 0 \\
			0 & 1 & 0 & 0 & b_5 & 0 & 0 & 0 & 0 \\
			0 & 0 & 1 & 0 & c_5 & 0 & 0 & 0 & 0 \\
			0 & 0 & 0 & 0 & 0 & 1 & 0 & 0 & 0 \\
			0 & -1 & 0 & 2 & 0 & 0 & a_2 & a_3 & 0 \\
			0 & 0 & 1 & 0 & -3 & -a_2 & 0 & 2 a_2+3 & 0 \\
			0 & 0 & 0 & 0 & 0 & -b_5 & 1 & 2 b_5 & 0 \\
			0 & 0 & 0 & 0 & 0 & -c_5 & 0 & 2 c_5+1 & 0 \\
			0 & 0 & 0 & 0 & 0 & 0 & 0 & 0 & 1 
	\end{pmatrix},
\end{align*}
and it has  rank 8 if 
\begin{itemize}
\item $b_5=2a_2$ and $c_5=-3$, or
\item $a_2=1$, $a_3=4$, and $c_5=-\frac{1}{2}$, or 
\item $a_2\neq1,b_5=\frac{-8 a_2^2-12 a_2+5 a_3}{6(a_2-1)}$, and $c_5=-\frac{1}{2}$,
\end{itemize} giving a total of $p^2+p+p(p-1)=2p^2$ possibilities.

Let \begin{align*}I_5=\begin{pmatrix}
		1&a_2&a_3&0&0&0\\
		0&0&0&1&0&0\\
		0&1&b_3&0&0&0\\
		0&0&0&0&1&0\\
		0&0&0&0&0&1
	\end{pmatrix}.
\end{align*}
We have 
\begin{align*}
	\phi(I_5)=\begin{pmatrix}
			1 & a_2 & a_3 & 0 & 0 & 0 & 0 & 0 & 0 \\
			0 & 0 & 0 & 1 & 0 & 0 & 0 & 0 & 0 \\
			0 & 1 & b_3 & 0 & 0 & 0 & 0 & 0 & 0 \\
			0 & 0 & 0 & 0 & 1 & 0 & 0 & 0 & 0 \\
			0 & 0 & 0 & 0 & 0 & 1 & 0 & 0 & 0 \\
			0 & -1 & 0 & 2 & 0 & 0 & a_2 & a_3 & 0 \\
			0 & 0 & 1 & 0 & -3 & 0 & 0 & 3 & 0 \\
			0 & 0 & 0 & 0 & 0 & 0 & 1 & b_3 & 0 \\
			0 & 0 & 0 & 0 & 0 & -1 & 0 & 2 & 0 \\
			0 & 0 & 0 & 0 & 0 & 0 & 0 & 0 & 1 
	\end{pmatrix},
\end{align*}
and none of them has rank 8. 

Finally, let \begin{align*}I_5=\begin{pmatrix}
		1&a_2&0&0&0&0\\
		0&0&0&1&0&0\\
		0&0&1&0&0&0\\
		0&0&0&0&1&0\\
		0&0&0&0&0&1
	\end{pmatrix}.
\end{align*}
We have 
\begin{align*}
	\phi(I_5)=\begin{pmatrix}
			1 & a_2 & 0 & 0 & 0 & 0 & 0 & 0 & 0 \\
			0 & 0 & 0 & 1 & 0 & 0 & 0 & 0 & 0 \\
			0 & 0 & 1 & 0 & 0 & 0 & 0 & 0 & 0 \\
			0 & 0 & 0 & 0 & 1 & 0 & 0 & 0 & 0 \\
			0 & 0 & 0 & 0 & 0 & 1 & 0 & 0 & 0 \\
			0 & -1 & 0 & 2 & 0 & 0 & a_2 & 0 & 0 \\
			0 & 0 & 1 & 0 & -3 & 0 & 0 & 3 & 0 \\
			0 & 0 & 0 & 0 & 0 & 0 & 0 & 1 & 0 \\
			0 & 0 & 0 & 0 & 0 & -1 & 0 & 2 & 0 \\
			0 & 0 & 0 & 0 & 0 & 0 & 0 & 0 & 1 
	\end{pmatrix},
\end{align*}
and it always has rank 8, giving $p$ possibilities. 

To summarize, we have
\begin{align*}
	b_{0,0,0,1,5,8}^{2,1}(L)&=\begin{cases}
		p^3+2p^2&\textrm{if }p\equiv 3,5\mod 8,\\
		p^3+2p^2-2p&\textrm{if }p\equiv 1,2\mod 8,
	\end{cases}.
\end{align*}

\subsubsection{Case 2-2: $I_4=\begin{pmatrix}
		0&1&a_3
	\end{pmatrix}$}

Let $b_{\bfm}^{2,2}(L)$ denote the number of graded ideal $I\idealgr L$ of dimension $\bfm$ where \[I_4=\begin{pmatrix}
	0&1&a_3
\end{pmatrix}.\] Note that there are $p$ such $I_4$ in $L_4$. This gives
\begin{align*}\phi(I_4)&=\begin{pmatrix}
		0&1&a_3&0&0&0\\
		0&0&0&0&1&a_3
	\end{pmatrix}.
\end{align*}
There are four possible forms of $I_5$, namely
\begin{align*}\begin{pmatrix}
		0&1&a_3&0&0&0\\
		0&0&0&0&1&a_3\\
		1&0&0&0&0&b_6\\
		0&0&1&0&0&c_6\\
		0&0&0&1&0&e_6
	\end{pmatrix},&
	\begin{pmatrix}
		0&1&a_3&0&0&0\\
		0&0&0&0&1&0\\
		1&0&0&b_4&0&0\\
		0&0&1&c_4&0&0\\
		0&0&0&0&0&1
	\end{pmatrix},\\
	\begin{pmatrix}
		0&1&a_3&0&0&0\\
		0&0&0&0&1&0\\
		1&0&b_3&0&0&0\\
		0&0&0&1&0&0\\
		0&0&0&0&0&1
	\end{pmatrix},&
	\begin{pmatrix}
		0&1&a_3&0&0&0\\
		0&0&0&0&1&0\\
		0&0&1&0&0&0\\
		0&0&0&1&0&0\\
		0&0&0&0&0&1
	\end{pmatrix},
\end{align*}
where $b_3,b_4,b_6,c_4,c_6,e_6\in\Fp$. 

Let \begin{align*}I_5=\begin{pmatrix}
		0&1&a_3&0&0&0\\
		0&0&0&0&1&a_3\\
		1&0&0&0&0&b_6\\
		0&0&1&0&0&c_6\\
		0&0&0&1&0&e_6
	\end{pmatrix}.
\end{align*}

We have
\begin{align*}
	\phi(I_5)=\begin{pmatrix}
			0 & 1 & a_3 & 0 & 0 & 0 & 0 & 0 & 0 \\
			0 & 0 & 0 & 0 & 1 & a_3 & 0 & 0 & 0 \\
			1 & 0 & 0 & 0 & 0 & b_6 & 0 & 0 & 0 \\
			0 & 0 & 1 & 0 & 0 & c_6 & 0 & 0 & 0 \\
			0 & 0 & 0 & 1 & 0 & e_6 & 0 & 0 & 0 \\
			0 & 0 & 0 & 0 & 0 & 0 & 1 & a_3 & 0 \\
			0 & 0 & 0 & 0 & 0 & -1 & 0 & 2 & a_3 \\
			0 & -1 & 0 & 2 & 0 & 0 & 0 & 0 & b_6 \\
			0 & 0 & 0 & 0 & 0 & 0 & 0 & 1 & c_6 \\
			0 & 0 & 1 & 0 & -3 & 0 & 0 & 3 & e_6 
	\end{pmatrix},
\end{align*}
One can check that for $p\geq5$, $p^2$ of them have rank 8.

Let \begin{align*}I_5=\begin{pmatrix}
		0&1&a_3&0&0&0\\
		0&0&0&0&1&0\\
		1&0&0&b_4&0&0\\
		0&0&1&c_4&0&0\\
		0&0&0&0&0&1
	\end{pmatrix}.
\end{align*}
We have
\begin{align*}
	\phi(I_5)=\begin{pmatrix}
		0 & 1 & a_3 & 0 & 0 & 0 & 0 & 0 & 0 \\
		0 & 0 & 0 & 0 & 1 & 0 & 0 & 0 & 0 \\
		1 & 0 & 0 & b_4 & 0 & 0 & 0 & 0 & 0 \\
		0 & 0 & 1 & c_4 & 0 & 0 & 0 & 0 & 0 \\
		0 & 0 & 0 & 0 & 0 & 1 & 0 & 0 & 0 \\
		0 & 0 & 0 & 0 & 0 & 0 & 1 & a_3 & 0 \\
		0 & 0 & 0 & 0 & 0 & -1 & 0 & 2 & 0 \\
		0 & -1 & b_4 & 2 & -3 b_4 & 0 & 0 & 3 b_4 & 0 \\
		0 & 0 & c_4 & 0 & -3 c_4 & 0 & 0 & 3 c_4+1 & 0 \\
		0 & 0 & 0 & 0 & 0 & 0 & 0 & 0 & 1 
	\end{pmatrix},
\end{align*}
and none of them has rank 8.

Let \begin{align*}I_5=\begin{pmatrix}
		0&1&a_3&0&0&0\\
		0&0&0&0&1&0\\
		1&0&b_3&0&0&0\\
		0&0&0&1&0&0\\
		0&0&0&0&0&1
	\end{pmatrix}.
\end{align*}
We have
\begin{align*}\begin{pmatrix}
	0 & 1 & a_3 & 0 & 0 & 0 & 0 & 0 & 0 \\
	0 & 0 & 0 & 0 & 1 & 0 & 0 & 0 & 0 \\
	1 & 0 & b_3 & 0 & 0 & 0 & 0 & 0 & 0 \\
	0 & 0 & 0 & 1 & 0 & 0 & 0 & 0 & 0 \\
	0 & 0 & 0 & 0 & 0 & 1 & 0 & 0 & 0 \\
	0 & 0 & 0 & 0 & 0 & 0 & 1 & a_3 & 0 \\
	0 & 0 & 0 & 0 & 0 & -1 & 0 & 2 & 0 \\
	0 & -1 & 0 & 2 & 0 & 0 & 0 & b_3 & 0 \\
	0 & 0 & 1 & 0 & -3 & 0 & 0 & 3 & 0 \\
	0 & 0 & 0 & 0 & 0 & 0 & 0 & 0 & 1 
\end{pmatrix}
\end{align*}
and none of them has rank 8.

Let \begin{align*}I_5=\begin{pmatrix}
		0&1&a_3&0&0&0\\
		0&0&0&0&1&0\\
		0&0&1&0&0&0\\
		0&0&0&1&0&0\\
		0&0&0&0&0&1
	\end{pmatrix}.
\end{align*}
We have
\begin{align*}
	\phi(I_5)=	\phi(I_5)=\begin{pmatrix}
		0 & 1 & a_3 & 0 & 0 & 0 & 0 & 0 & 0 \\
		0 & 0 & 0 & 0 & 1 & 0 & 0 & 0 & 0 \\
		0 & 0 & 1 & 0 & 0 & 0 & 0 & 0 & 0 \\
		0 & 0 & 0 & 1 & 0 & 0 & 0 & 0 & 0 \\
		0 & 0 & 0 & 0 & 0 & 1 & 0 & 0 & 0 \\
		0 & 0 & 0 & 0 & 0 & 0 & 1 & a_3 & 0 \\
		0 & 0 & 0 & 0 & 0 & -1 & 0 & 2 & 0 \\
		0 & 0 & 0 & 0 & 0 & 0 & 0 & 1 & 0 \\
		0 & 0 & 1 & 0 & -3 & 0 & 0 & 3 & 0 \\
		0 & 0 & 0 & 0 & 0 & 0 & 0 & 0 & 1 
	\end{pmatrix},
\end{align*}
and it always has rank 8. To summarize, we have
\begin{align*}
	b_{0,0,0,1,5,8}^{2,2}(L)&=p^2+p.
\end{align*}

\subsubsection{Case 2-3: $I_4=\begin{pmatrix}
		0&0&1
	\end{pmatrix}$}

Let $b_{\bfm}^{2,3}(L)$ denote the number of graded ideal $I\idealgr L$ of dimension $\bfm$ where \[I_4=\begin{pmatrix}
	0&0&1
\end{pmatrix}.\] Note that there are exactly 1 such $I_4$ in $L_4$. This gives
\begin{align*}\phi(I_4)&=\begin{pmatrix}
		0&0&1&0&0&0\\
		0&0&0&0&0&1
	\end{pmatrix}.
\end{align*}
Suppose $\dim(I_5)=5$. Here there are four possible forms of $I_5$, namely
\begin{align*}\begin{pmatrix}
		0&0&1&0&0&0\\
		0&0&0&0&0&1\\
		1&0&0&0&b_5&0\\
		0&1&0&0&c_5&0\\
		0&0&0&1&e_5&0
	\end{pmatrix},&
	\begin{pmatrix}
		0&0&1&0&0&0\\
		0&0&0&0&0&1\\
		1&0&0&b_4&0&0\\
		0&1&0&c_4&0&0\\
		0&0&0&0&1&0
	\end{pmatrix},\\
	\begin{pmatrix}
		0&0&1&0&0&0\\
		0&0&0&0&0&1\\
		1&b_2&0&0&0&0\\
		0&0&0&1&0&0\\
		0&0&0&0&1&0
	\end{pmatrix},&
	\begin{pmatrix}
		0&0&1&0&0&0\\
		0&0&0&0&0&1\\
		0&1&0&0&0&0\\
		0&0&0&1&0&0\\
		0&0&0&0&1&0
	\end{pmatrix},
\end{align*}
where $b_2,b_4,b_5,c_4,c_5,e_5\in\Fp$. 

Let \begin{align*}I_5=\begin{pmatrix}
		0&0&1&0&0&0\\
		0&0&0&0&0&1\\
		1&0&0&0&b_5&0\\
		0&1&0&0&c_5&0\\
		0&0&0&1&e_5&0
	\end{pmatrix}.
\end{align*}

We have
\begin{align*}
	\phi(I_5)=\begin{pmatrix}
			0 & 0 & 1 & 0 & 0 & 0 & 0 & 0 & 0 \\
			0 & 0 & 0 & 0 & 0 & 1 & 0 & 0 & 0 \\
			1 & 0 & 0 & 0 & b_5 & 0 & 0 & 0 & 0 \\
			0 & 1 & 0 & 0 & c_5 & 0 & 0 & 0 & 0 \\
			0 & 0 & 0 & 1 & e_5 & 0 & 0 & 0 & 0 \\
			0 & 0 & 0 & 0 & 0 & 0 & 0 & 1 & 0 \\
			0 & 0 & 0 & 0 & 0 & 0 & 0 & 0 & 1 \\
			0 & -1 & 0 & 2 & 0 & -b_5 & 0 & 2 b_5 & 0 \\
			0 & 0 & 0 & 0 & 0 & -c_5 & 1 & 2 c_5 & 0 \\
			0 & 0 & 1 & 0 & -3 & -e_5 & 0 & 2 e_5+3 & 0 
	\end{pmatrix},
\end{align*}
and none of them has rank 8.

Let \begin{align*}I_5=\begin{pmatrix}
		0&0&1&0&0&0\\
		0&0&0&0&0&1\\
		1&0&0&b_4&0&0\\
		0&1&0&c_4&0&0\\
		0&0&0&0&1&0
	\end{pmatrix}.
\end{align*}

We have
\begin{align*}
	\phi(I_5)=\begin{pmatrix}
			0 & 0 & 1 & 0 & 0 & 0 & 0 & 0 & 0 \\
			0 & 0 & 0 & 0 & 0 & 1 & 0 & 0 & 0 \\
			1 & 0 & 0 & b_4 & 0 & 0 & 0 & 0 & 0 \\
			0 & 1 & 0 & c_4 & 0 & 0 & 0 & 0 & 0 \\
			0 & 0 & 0 & 0 & 1 & 0 & 0 & 0 & 0 \\
			0 & 0 & 0 & 0 & 0 & 0 & 0 & 1 & 0 \\
			0 & 0 & 0 & 0 & 0 & 0 & 0 & 0 & 1 \\
			0 & -1 & b_4 & 2 & -3 b_4 & 0 & 0 & 3 b_4 & 0 \\
			0 & 0 & c_4 & 0 & -3 c_4 & 0 & 1 & 3 c_4 & 0 \\
			0 & 0 & 0 & 0 & 0 & -1 & 0 & 2 & 0
	\end{pmatrix},
\end{align*}
and it has rank 8 if $c_4=-2$, giving $p$ possibilities.

Let \begin{align*}I_5=\begin{pmatrix}
		0&0&1&0&0&0\\
		0&0&0&0&0&1\\
		1&b_2&0&0&0&0\\
		0&0&0&1&0&0\\
		0&0&0&0&1&0
	\end{pmatrix}
\end{align*}

We have
\begin{align*}
	\phi(I_5)=\begin{pmatrix}
			0 & 0 & 1 & 0 & 0 & 0 & 0 & 0 & 0 \\
			0 & 0 & 0 & 0 & 0 & 1 & 0 & 0 & 0 \\
			1 & b_2 & 0 & 0 & 0 & 0 & 0 & 0 & 0 \\
			0 & 0 & 0 & 1 & 0 & 0 & 0 & 0 & 0 \\
			0 & 0 & 0 & 0 & 1 & 0 & 0 & 0 & 0 \\
			0 & 0 & 0 & 0 & 0 & 0 & 0 & 1 & 0 \\
			0 & 0 & 0 & 0 & 0 & 0 & 0 & 0 & 1 \\
			0 & -1 & 0 & 2 & 0 & 0 & b_2 & 0 & 0 \\
			0 & 0 & 1 & 0 & -3 & 0 & 0 & 3 & 0 \\
			0 & 0 & 0 & 0 & 0 & -1 & 0 & 2 & 0 
	\end{pmatrix},
\end{align*}
and it always has rank 8, giving $p$ possibilities.

Let \begin{align*}I_5=\begin{pmatrix}
		0&0&1&0&0&0\\
		0&0&0&0&0&1\\
		0&1&0&0&0&0\\
		0&0&0&1&0&0\\
		0&0&0&0&1&0
	\end{pmatrix}
\end{align*}

We have
\begin{align*}
	\phi(I_5)=\begin{pmatrix}
			0 & 0 & 1 & 0 & 0 & 0 & 0 & 0 & 0 \\
			0 & 0 & 0 & 0 & 0 & 1 & 0 & 0 & 0 \\
			0 & 1 & 0 & 0 & 0 & 0 & 0 & 0 & 0 \\
			0 & 0 & 0 & 1 & 0 & 0 & 0 & 0 & 0 \\
			0 & 0 & 0 & 0 & 1 & 0 & 0 & 0 & 0 \\
			0 & 0 & 0 & 0 & 0 & 0 & 0 & 1 & 0 \\
			0 & 0 & 0 & 0 & 0 & 0 & 0 & 0 & 1 \\
			0 & 0 & 0 & 0 & 0 & 0 & 1 & 0 & 0 \\
			0 & 0 & 1 & 0 & -3 & 0 & 0 & 3 & 0 \\
			0 & 0 & 0 & 0 & 0 & -1 & 0 & 2 & 0 
	\end{pmatrix},
\end{align*}
and it always has rank 8, giving 1 possibility. To summarize, we have
\begin{align*}
	b_{0,0,0,1,5,8}^{2,3}(L)&=2p+1,
\end{align*}
and
\begin{align*}
	b_{0,0,0,1,5,8}(L)&=\begin{cases}
		p^3+3p^2+3p+1&\textrm{if }p\equiv 3,5\mod 8,\\
		p^3+3p^2+p+1&\textrm{if }p\equiv 1,7\mod 8,
	\end{cases}.
\end{align*}

To conclude, we have

\begin{align*}
	b_{0,0,0,2,4,8}(L)&=p^2+4p+1,\\
	b_{0,0,0,2,5,7}(L)&=0,\\
	b_{0,0,0,1,4,9}(L)&=p^6+2p^5+4p^4+4p^3+4p^2+2p+1,\\
	b_{0,0,0,1,5,8}(L)&=\begin{cases}
		p^3+3p^2+3p+1&\textrm{if }p\equiv 3,5\mod 8,\\
		p^3+3p^2+p+1&\textrm{if }p\equiv 1,7\mod 8,\end{cases},\\
	b_{0,0,0,0,5,9}(L)&=p^5+p^4+p^3+p^2+p+1,	
\end{align*}
which gives
\begin{align*}
		a_{p^{9}}^{\idealgr}(L)&=b_{0,0,0,2,4,8}(L)+b_{0,0,0,2,5,7}(L)+b_{0,0,0,1,4,9}(L)+b_{0,0,0,1,5,8}(L)+b_{0,0,0,0,5,9}(L)\\
		&=\begin{cases}
			p^6+3p^5+5p^4+6p^3+9p^2+10p+4&p\equiv3,5\mod8,\\
			p^6+3p^5+5p^4+6p^3+9p^2+8p+4&p\equiv1,7\mod8,
		\end{cases}
\end{align*}
as required. In particular, $a_{p^{9}}^{\idealgr}(L)$ is not uniformly given by a polynomial in $p$. Hence $\zeta_{\mff_{6,2}(\Fp)}^{\idealgr}(s)$ is not $\Fp$-uniform.
\end{proof}

\begin{rem}
	As mentioned at the end of Section \ref{sec:nsg}, $b_{0,0,0,1,5,8}(L)$ is one of the constituents of $a_2(n,4,p)(\mff_{c,2}(\Fp))$ for $c\geq6$, which stops $\zeta_{\mff_{6,2}}^{\idealgr}(s)$ to be $\Fp$-uniform. Unfortunately, at the current stage it looks out of reach to fully compute  $a_2(n,4,p)(\mff_{6,2}(\Fp))$ and $\zeta_{\mff_{6,2}(\Fp)}^{\idealgr}(s)$. In particular, we do not even know whether $\zeta_{\mff_{6,2}}^{\idealgr}(s)$ is just $\Fp$-finitely uniform, or even $\Fp$-non-uniform.
\end{rem}

\section{Further questions}\label{sec:further}
\subsection{More computations of $\mff_{c,d}(\Fp)$}
It is interesting that the $\Fp$-uniformity of  $\zeta_{\mff_{c,2}}^{\idealgr}(s)$ fails at $c=6$, whereas it is known that  $\zeta_{\mff_{2,d}}^{\ideal}(s)$ is $\Fp$-uniform (\cite[Theorem 5.1]{Lee/20arxiv1}) and $\Zp$-uniform (\cite[Theorem 2]{GSS/88}) for all $d\in\N$. Although we do not know $\zeta_{\mff_{c,2}(\Fp)}^{\idealgr}(s)$ for $c\geq7$ yet, the authors believe it would be really unlikely to be uniformly given by a polynomial in $p$.  
\begin{con}
	$\zeta_{\mff_{c,2}}^{\idealgr}(s)$ is not $\Fp$-uniform for all $c\geq6$.
\end{con}

Also note that we could not fully compute $\zeta_{\mff_{6,2}(\Fp)}^{\idealgr}(s)$ yet. Is $\zeta_{\mff_{6,2}}^{\idealgr}(s)$ still $\Fp$-finitely uniform? Or would it be simply $\Fp$-non-uniform? How about ungraded examples? Given that the counting problems arising from the zeta functions already became complicated for $c=6$, the authors believe the following naive conjecture:
\begin{con}
	There exists $c\in\N$ such that $\zeta_{\mff_{c,2}}^{\idealgr}(s)$ is $\Fp$-non-uniform.
\end{con}

In addition to our work, one can naturally ask what happens for a larger number of generators. 
\begin{qun}
	For $c,d\geq3$, can we find more pairs of $(c,d)$ such that $\zeta_{\mff_{c,d}}^{\idealgr}(s)$ or $\zeta_{\mff_{c,d}}^{\ideal}(s)$ is not $\Fp$-uniform?
\end{qun}
Unfortunately we do not know $\zeta_{\mff_{c,d}(\Fp)}^{\idealgr}(s)$ or $\zeta_{\mff_{c,d}(\Fp)}^{\ideal}(s)$ for $c,d\geq3$ yet. It would be very useful to extend the database.

\subsection{$\Zp$- and $\Fp$-uniformity}
All the work in this article concerns zeta functions of Lie algebras defined over $\Fp$ rather than over $\Zp$. In many ways this simplifies matters considerably -- we ``merely'' have to evaluate a finite sum. How far one can extend? For example, it is proven in \cite{LeeVoll/18} that $\zeta_{\mff_{c,2}}^{\idealgr}(s)$ is $\Zp$-uniform for $c\leq 4$, which is compatible with our result. In this light, one can naively guess that $\zeta_{\mff_{5,2}}^{\idealgr}(s)$ is $\Zp$-uniform but $\zeta_{\mff_{6,2}}^{\idealgr}(s)$ is not. Would it be actually true?

More generally, as discussed by the second author in \cite{Lee/20arxiv1}, there exist finitely many varieties $U_{1}^{\ideal},\ldots,U_{h}^{\ideal}$ defined over $\mathbb{Q}$,  and polynomials $W_{1}^{\ideal}(X,Y),\ldots,W_{h}^{\ideal}(X,Y)\in\mathbb{Z}[X,Y]$ such that, for almost all primes $p$, 
	\begin{equation*}
		\zeta_{\mff_{c,d}(\Fp)}^{\ideal}(s)=\sum_{j=1}^{h}|\overline{U_{j}^{\ideal}}(\Fp)|W_{j}^{\ideal}(p,t).
	\end{equation*} 
Can we relate the varieties arising from $\zeta_{\mff_{c,d}(\Fp)}^{\ideal}(s)$ to the varieties arising from 	$\zeta_{\mff_{c,d}(\Zp)}^{\ideal}(s)$? In particular:
\begin{qun}
	For any $\mff_{c,d}$, does the $\Zp$-uniformity of $\zeta_{\mff_{c,d}}^{\ideal}(s)$ always equal to its $\Fp$-uniformity?  
\end{qun} 
One way to see this problem is the following. To construct a Lie ring $L$ such that $\zeta_{\mff_{c,d}}^{\ideal}(s)$ is $\Fp$-uniform but not $\Zp$-uniform,  one approach would be to find a variety $V$ over
$\mathbb Z$ defined by a polynomial equation for which the number of points
on the reduction $V(\mathbb F_p)$ is uniformly given by a polynomial in $p$
but for which the number of points on the variety $V({\mathbb Z}/{p^n\mathbb
	Z})$ is not for some $n>1.$ This is a ``Hensel's Lemma type" problem. We know
that smooth points on a variety will lift in a well behaved uniform way from
$\mathbb F_p$ and so we will need a to produce a uniform variety with
non-uniform singular set. Then we will need to encode this variety in a
presentation for a Lie ring. It would be really interesting to check whether such encoding can be observed in free Lie rings. At present this remains out of reach.

\subsection{Amenable representations for free Lie algebras}
The other questions that have arisen from our approach relate to methods for 
producing presentations for the free Lie algebras. The
classical solution to this problem is to generate a Hall Basis for each layer of 
the lower central series. For our purposes this presentation
is somewhat intractable. In trying to understand the dimension of collapse 
of a subspace $I_k$ of the $k$-th term of the lower central
series of a free Lie algebra $L$ we make use of the maps $\phi_i: L_k 
\rightarrow L_{k+1}$. To understand the relevant properties of
these maps in a $d$-generated algebra is closely related to recursively 
producing a list of generators for $L_k$ consisting of
Lie elements of the form
\[
[X_i, w_{k-1}], \quad w_{k-1} \in L_{k-1}.
\]
The algorithm for a Hall Basis does not deliver generators of this form. One can naively ask the following question:
\begin{qun}
	Is there any general algorithm that gives a basis in the form we would like?
\end{qun}

\section*{Acknowledgments}
The second author is supported by the National Research Foundation of Korea (NRF) grant funded by the Korean government (MEST), No. 2019R1A6A1A10073437. The authors greatly thank Cornelius Griffin and Fritz Grunewald for their contribution on the earlier version of this work. The authors gratefully acknowledge inspiring mathematical discussions with Seok Hyeong Lee and Michael Vaughan-Lee.

\bibliographystyle{amsplain}
\bibliography{duSL_Uniformity}
\end{document}